\documentclass[EJP, preprint]{ejpecp} % replace ECP by EJP if needed.
\usepackage{amsmath}
\usepackage{amssymb}
\usepackage{amsthm}
\usepackage{mathrsfs} 
\usepackage{bbm}
\usepackage{bbold}
\usepackage{mathtools}
\usepackage{xcolor}
\usepackage{subcaption}
\usepackage{enumitem}
\usepackage{soul}
\usepackage{bm}
\usepackage{orcidlink}

\DeclareMathOperator{\sign}{sign}
\theoremstyle{plain}

\SHORTTITLE{Construction and limit theorems for supCAR fields}

\TITLE{Construction and limit theorems for supCAR fields} % \thanks is optional. Insert line breaks with \\

\AUTHORS{%
  Illia~Donhauzer\footnote{La Trobe University, Australia.
    \EMAIL{i.donhauzer@latrobe.edu.au}}\orcid{0000-0002-4399-2631}
  \and %% remove this line and below if single author
  Nikolai~Leonenko\footnote{Cardiff University, UK. \BEMAIL{LeonenkoN@cardiff.ac.uk}}\orcid{0000-0003-1932-4091}
 \and %% remove this line and below if single author
  Andriy~Olenko\footnote{La Trobe University, Australia. \BEMAIL{a.olenko@latrobe.edu.au}}\orcid{0000-0002-0917-7000}
  }%AUTHORS
 % Separate items with ;

\AMSSUBJ{60G60, 60F05, 60G10} % Edit. Separate items with ;
\SUBMITTED{June 1, 2025} % Edit.

\VOLUME{0}
\YEAR{2025}
\PAPERNUM{0}
\DOI{10.1214/YY-TN}

\ABSTRACT{The paper introduces a new class of random fields, supCAR~fields, which are constructed as superpositions of continuous autoregressive random fields. These supCAR fields possess infinitely divisible marginal distributions. Their second-order properties are characterised by a novel family of covariance functions which can exhibit short- and long-range spatial dependencies. First, the existence of such fields is examined. Then, functional limit theorems for supCAR fields are derived under general assumptions.  Four limiting scenarios that depend on the marginals of the underlying autoregressive fields and the specifications of the superposition are identified. Examples of specific supCAR fields, for which the assumptions and results are provided in simple, explicit forms, are presented. The obtained limit theorems can be employed for the statistical inference of supCAR~fields.}

\begin{document}

%%%%%%%%%%%%%%%%%%%%%%%%%%%%%%%%%%%%%%%%%%%%%%%%%%%%%%%%%%%%%%%%%%%
%%                                                               %%
%% No need for \maketitle.                                       %%
%%                                                               %%
%%%%%%%%%%%%%%%%%%%%%%%%%%%%%%%%%%%%%%%%%%%%%%%%%%%%%%%%%%%%%%%%%%%

%%%%%%%%%%%%%%%%%%%%%%%%%%%%%%%%%%%%%%%%%%%%%%%%%%%%%%%%%%%%%%%%%%%
%%                                                               %%
%% Please replace what follows by the body of your article       %%
%% (up to the bibliography):                                     %%
%%                                                               %%
%%%%%%%%%%%%%%%%%%%%%%%%%%%%%%%%%%%%%%%%%%%%%%%%%%%%%%%%%%%%%%%%%%%

\begin{center}
{\it In memory of P. Brockwell (1937-2023) and O.E. Barndorff-Nielsen (1935-2022)}
\end{center}

\section{Introduction}

This article commemorates the pioneering contributions in stochastics of P. Brockwell and O.E.~Barndorff-Nielsen and develops a new model of supCAR random fields based on the approaches they suggested. 

This paper focuses on a modelling framework based on the so-called ambit fields~\cite{BarndorffNielsen2018}, which have first been proposed in the context of modelling turbulence in physics, and have been further developed in applications ranging from tumour growth in biology to financial markets and energy systems, see~\cite{barndorff1997processes, BarndorffNielsen2018, bnleonenko2005, barndorff2001non}. During the last decade, Ambit Stochastics has established itself as a new research area of probability theory and statistics, complementing and extending research on L\'{e}vy processes and infinitely divisible distributions.

An ambit field is defined as a stochastic integral of a deterministic kernel function with respect to an infinitely divisible random measure. Special cases of ambit fields include purely temporal processes (such as Ornstein-Uhlenbeck~(OU) processes, continuous-time autoregressive moving average processes, (mixed) moving average processes, supOU processes and trawl processes) and purely spatial processes (such as Gaussian fields with, for instance, the Mat\'{e}rn correlation function).

Many fundamental results in the contemporary theory of random processes and fields are established under the assumptions of Gaussianity and fast decay of tails of marginal distributions, see, for example, the results on geometric properties of Gaussian fields~\cite{adler1981random, adler2009random}, asymptotic analysis of their functionals \cite{leonenko1999}, and their extreme behaviour~\cite{adler2009random, azais2009level}.
The class of infinitely divisible random fields generalises the class of Gaussian fields. It includes fields with marginal distribution tails decaying more slowly than exponentially, in particular, fields with infinite mean and variance. Infinitely divisible fields may exhibit properties that are not typical for Gaussian ones. The geometry of smooth infinitely divisible fields differs significantly from that of smooth Gaussian fields. For instance, excursion sets above high levels for smooth Gaussian fields are, with high probability, approximately ellipsoidal, while for smooth infinitely divisible fields, excursion sets can have almost any shape~\cite{adler2013high}. More intricate properties of infinitely divisible fields lead to limit theorems with distributions that are not typical for Gaussian models \cite{puplinskaite2016aggregation}, making a statistical inference for infinitely divisible fields more complex.

One of the approaches to obtain infinitely divisible fields was recently introduced in~\cite{brockwell2017continuous}, where the authors constructed continuous autoregressive moving average random
fields on $\mathbb{R}^d$ (CARMA($p$,$q$) fields, where $p$ is the order of autoregression and $q$ is the order of moving average) as stochastic integrals of CARMA kernels. CARMA fields are the multidimensional continuous-time analogues of ARMA time series, possessing infinitely divisible marginal distributions and a large family of covariance functions. In its simplest form with $p=1$ and $q=0$, CARMA fields become continuous autoregressive fields (CAR(1)) defined as
\[ S_\lambda(\bm{t}) := -\int\limits_{\mathbb{R}^d} \frac{1}{2\lambda} e^{-\lambda||\bm{t}-\bm{u}||}d{L(\bm{u})}, \ \bm{t}\in\mathbb{R}^d, \] where $L(\cdot)$ is a L\'evy sheet and $\lambda>0$ is fixed. CAR(1) fields possess the Mat\'ern covariances scaled by $\lambda.$ Their infinitely-divisible marginal distributions depend on the choice of $L(\cdot).$

This paper introduces a new class of ambit fields with a rich class of marginal distributions and flexible dependence structures. It is called supCAR random fields, since their kernel is taken from CARMA random fields. We exploit the idea of superpositions of stochastic processes in $\mathbb{R}$ proposed in the seminal paper by Barndorff-Nielsen \cite{barndorff2001superposition}, where superpositions of independent Ornstein–Uhlenbeck type processes (SupOU processes) have provided flexible and analytically tractable parametric models.   This paper studies the following superpositions defined as Rajput-Rosinski integrals \cite{rajput1989spectral} of CAR(1) kernels 
\[ X(\bm{t}) := -\int\limits_0^\infty\int\limits_{\mathbb{R}^d} \frac{1}{2\lambda} e^{-\lambda||\bm{t}-\bm{u}||}\Lambda(d\bm{u}, d\lambda),\ \bm{t}\in\mathbb{R}^d, \] where $\Lambda(\cdot)$ is an independently scattered scalar infinitely divisible random measure.
Unlike the Ornstein–Uhlenbeck kernels used in supOU processes, CAR(1) kernels are more complex and can have singularities at the origin. Thus, while the considered supCAR fields have some similar properties to supOU processes, they are not their direct generalisations but form a distinct class of models with novel local and asymptotic properties. One of the main tools is the obtained representation for supCAR's joint cumulant~function. 

An important feature of the introduced supCAR fields is that their dependence structures are determined by the specifications of the superposition and do not depend on the marginal distributions of the underlying CAR(1) fields. This allows the marginal distributions and covariances of the supCAR fields to be specified separately.
Moreover, supCAR fields exhibit various spatial dependence structures governed by a new, rich class of covariance functions. The corresponding covariance functions are not necessarily non-negative and not necessarily monotonically decreasing. By specifying the superpositions, one can obtain the desired local behaviour of a supCAR field, making its realisations smoother or rougher, while simultaneously ensuring that the large-scale behaviour exhibits either short- or long-range dependence.

The asymptotic behaviour of supCAR fields is studied under general conditions. We focus on functional limit theorems and use the cumulant method. The established limit theorems cover both weak and long-range dependence regimes. The obtained limit theorems could be viewed as complementary to the classical limit theorems for stochastic processes and fields, particularly those involving Gaussian subordinated models, consult~\cite{leonenko1999, leonenko2014sojourn, OlenkoOmari2020, ta1,ta2} and the references therein on asymptotics under weak and long-range dependence regimes. Within the considered framework, there exist interesting limiting scenarios, particularly in the context of the long-range dependence regime, that did not appear in the aforementioned publications. 
Contrary to the case of integrated Gaussian fields, for which two possible asymptotic scenarios are the Brownian motion and generalised Brownian motion, for integrated supCAR fields there are four limit scenarios (Brownian motion, generalised Brownian motion, $\alpha$-stable Levy process, and $\gamma$-stable process). The results generalise the recent findings for superpositions of Ornstein-Uhlenbeck-type processes in \cite{grahovac2019limit} to the case of random fields and their averages over arbitrary multidimensional sets. It is demonstrated how the limit distributions depend on the considered sets, settings of superpositions (short- or long-range dependence of supCAR fields) and marginal distributions of the underlying CAR(1) processes. Compared to the case of supOR processes, the study of supCAR random fields and the corresponding functionals requires different assumptions and addressing specific challenges arising from the multidimensional setting.

The paper is structured as follows. Section~\ref{Preliminaries} provides the main definitions and notations. Section~\ref{sec:supcar} introduces the superpositions of continuous autoregressive fields. Then, this section
investigates the conditions for the existence of supCAR fields, their finite-dimensional distributions, and dependence structures. This section also derives closed-form representations for their covariance functions and spectral densities. Section~\ref{sect:limit} is dedicated to the examination of limit theorems for integrated supCAR fields. The obtained theoretical findings are illustrated via simulation studies of Gamma fields.

All numerical computations, simulations and plotting in this paper were performed using the software R (version 4.5.0) and Python (version 3.10.10). The corresponding code is freely available in the folder
”Research materials” from the website \url{https://sites.google.com/site/olenkoandriy/}.

\section{Preliminaries}\label{Preliminaries}
Throughout the paper, bold letters will be used to denote vectors (e.g. $\textit{\textbf{x}}$), while numbers and scalar variables will be denoted in a regular font.   In what follows, $||\cdot||$ is the Euclidean norm in $\mathbb{R}^d,$ $|\cdot|$ and $Leb(\cdot)$ denotes the Lebesgue measure on $\mathbb{R}^d$, and $\sigma(\mathbb{R}^d)$ stands for the $\sigma$-algebra of Borel sets on $\mathbb{R}^d.$ $B_\textit{\textbf{x}}(r)=\{\textit{\textbf{y}}\in\mathbb{R}^d: \ ||\textit{\textbf{x}}-\textit{\textbf{y}}||\leq r\}$ denotes a ball in $\mathbb{R}^d$ with the center at $\textit{\textbf{x}}$ and the radius $r\geq0.$ If the center is at the origin $\bm{x}=\bm{0},$ we use the simplified notation $B(r).$ $\mathcal{F}_d$ denotes $d$-dimensional Fourier transform. Throughout the paper, $\mathcal{S}$ is a $\sigma$-ring (a ring which is closed under the countable intersection) of the subsets of $\mathbb{R}^d$ with the property that there exists an increasing sequence $\{ S_n \}$  of sets in $\mathcal{S}$ such\ that $\cup_n S_n=\mathbb{R}^d,$ while the $\sigma$-ring of subsets of $\mathbb{R}^+\times\mathbb{R}^d$ with the same property is denoted by $\mathcal{S}'.$ $\sigma(\mathcal{S})$ and $\sigma(\mathcal{S}')$ denote the smallest $\sigma$-algebras containing $\mathcal{S}$ and $\mathcal{S}'$ respectively. The notations $\overset{d}{\to}$ and $\overset{fdd}{\to}$ are used for the convergence in distribution and the convergence of finite-dimensional distributions, respectively. 

Most of the following notations related to infinite divisibility are adapted from \cite{barndorff2001superposition}. In the following, $C$ represents a finite positive constant, which is not necessarily the same in each expression.

A random variable $\xi$ is infinitely divisible if its cumulant function has the following L\'evy-Khinchine representation
\[ C(s\ddagger\xi) := \log Ee^{is \xi} = ias-\frac{b}{2}s^2 +\int\limits_\mathbb{R} (e^{is x}-1-is\tau(x))W(dx),\ s\in\mathbb{R},\] where $a\in\mathbb{R},\ b\geq0,$ and 
\begin{equation}\label{tau} \tau(x) :=
    \begin{cases}
      x, \ |x|\leq1,\\
      \sign(x), \ |x|>1,
    \end{cases}\
\end{equation} where $\sign(\cdot)$ is the signum function. The L\'evy measure $W(\cdot)$ is a Radon measure on $\mathbb{R}$ such that $W(\{0\})=0$ and 
\begin{equation}\label{eq:W_int} \int\limits_\mathbb{R}\min(1, x^2) W(dx) <\infty. \end{equation}  

The Blumental-Getoor index \cite{blumenthal1960some} of the L\'evy measure $W(\cdot)$ is defined as
\begin{equation*} \beta_{BG}:=\inf\left\{ \gamma\geq0:\int_{|x|\leq1}|x|^\gamma W(dx) <\infty\right\}.\end{equation*} The Blumental-Getoor index satisfies the condition $0\le \beta_{BG}\le 2,$ and for the L\'evy measure $W(\cdot)$ which index is $\beta_{BG},$ for $\alpha>\beta_{BG}$ it holds true that
\[ \int\limits_{|x|\leq1}|x|^{\alpha}W(dx)<\infty.\]

An infinitely divisible random variable $\xi$ is called self-decomposable if for each $c\in(0,1)$ there exists a characteristic function $\phi_c(\cdot)$ such that the characteristic function $\phi(\cdot)$ of $\xi$ can be represented as $\phi(s) = \phi(cs)\phi_c(s)$ for all $s\in\mathbb{R}.$ The equivalent definition  is that $\xi$ has a L\'evy measure $W(\cdot)$ of the form $W(dx) = w(x)dx, $ with $w(x)=|x|^{-1}\overline{w}(x),$ where the function $\overline{w}(x)$ increases on $(-\infty,0)$ and decreases on $(0,\infty).$ The class of self-decomposable distributions includes, for instance, Gamma, variance Gamma, inverse Gaussian, normal inverse Gaussian, and Student distributions.

 Let $\Lambda_0(\cdot)$ be an independently scattered scalar random measure on the $\sigma$-ring $\mathcal{S}$ of $\mathbb{R}^d$ in the sense that for every sequence of disjoint sets $\{A_n\}\in \mathcal{S}$ the random variables $\Lambda_0(A_n),\ n=1,2,...,$ are independent, and if $\cup_{n}A_n\in\mathcal{S}$ it holds
\[ \Lambda_0\left(\bigcup\limits_{n} A_n\right) = \sum\limits_{n}\Lambda_0(A_n) \ {\rm a.s.}\] Let $\Lambda_0(\cdot)$ also possesses infinitely divisible distributions such that for all  $A\in \mathcal{S}$ with finite Lebesgue measure $|A|,$ it holds
\begin{equation*}C(s\ddagger\Lambda_0(A))=im_0(A)s -\frac{m_1(A)}{2}s^2 +\int\limits_{\mathbb{R}}(e^{is x}-1-is \tau(x))Q(A, dx), \end{equation*}where $m_0(\cdot)$ is a signed measure, $m_1(\cdot)$ is a measure, and for a fixed $A$ the function $Q(A,\cdot)$ is a Radon measure on $\mathbb{R}.$ 

For a simple function $f(x)=\sum_{i=1}^n a_i \mathbb{1}_{A_i}(x), \ x \in\mathbb{R}^d,$ where $\mathbb{1}_{A_i}(\cdot)$ is an indicator function of the set $A_i\in\mathcal{S}, \ A_i\cap A_j=\emptyset,\ i\neq j,$ the Rajput-Rosinski integral \cite{rajput1989spectral} with respect to the independently scattered scalar infinitely divisible measure $\Lambda_0(\cdot)$ is defined~as
\[ \int\limits_A f(x) \Lambda_0(dx) := \sum\limits_{i=1}^n a_j \Lambda_0(A_i\cap A),\ A\in\mathcal{S}.\]  A function $f$ is $\Lambda_0$-integrable, if there exists a sequence of simple functions $\{ f_n \}$ such that $f_n\to f$ almost everywhere, and for every $A\in \sigma(\mathcal{S}),$ the sequence $\{ \int_A f_n(dx) \Lambda_0(dx) \}$ converges in probability. If $f$ is $\Lambda_0$-integrable, then its Rajput-Rosinski integral is defined~as
\[ \int\limits_Af(x)\Lambda_0(dx)  :=\lim\limits_{n\to\infty}\int\limits_Af_n(x)\Lambda_0(dx),\] where the limit is in the probability sense. The important property of the Rajput-Rosinski integrals is that under some weak conditions on the measure $\Lambda_0(\cdot),$ see \cite{rajput1989spectral} and Proposition \ref{prop:cum_func},
\[ \int\limits_{A}  f(\textit{\textbf{x}}) \Lambda_0(d\bm{x}) \in \overline{lin}\{\Lambda_0(A), A\in\mathcal{S}\}_{L_p},\] where the last denotes the closure in $L_p$ norm of the space spanned by the family of random variables $\Lambda_0(A), A\in\mathcal{S}.$

Continuous autoregressive fields were recently introduced in {{\rm \cite{brockwell2017continuous}}}, and this paper studies a new model based on their transformations. Let the cumulant function of the infinitely divisible independently scattered random measure $\Lambda_0(\cdot)$ allow the next factorisation
\begin{equation*} C(s\ddagger \Lambda_0(A)) = |A|\cdot \mathcal{K}(s),\ A\in\mathcal{S},\end{equation*} where $\mathcal{K}(\cdot)$ is a cumulant function of an infinitely divisible distribution and \begin{equation}\label{musigma}
E\Lambda_0(A) = \mu |A|, \ \ \ Var(\Lambda_0(A)) = \sigma^2|A|, \ \ \ \mu\in\mathbb{R}, \ \sigma^2\in\mathbb{R}^+.
\end{equation} It should be noted that the parameters $\mu$ and $\sigma^2$ are not the mean and variance of the Gaussian component of the measure $\Lambda_0(\cdot),$ they also depend on the L\'evy measure $W(\cdot).$

The scalar L\'evy sheet $L=\{L(\bm{t}),\ \bm{t}\in\mathbb{R}^d_+  \}$ associated with  the measure $\Lambda_0(\cdot)$ is defined~as 
\begin{equation}\label{Levyfield}L(\bm{t}):=\Lambda_0((\bm{0},\bm{t}]),\ \bm{t}\in \mathbb{R}^d_+,\end{equation} and the corresponding continuous autoregressive random fields is defined as follows.

\begin{theorem} {{\rm \cite{brockwell2017continuous}}}
Let $L(\cdot)$ be the second-order L\'evy sheet with parameters $\mu$ and $\sigma^2$ defined in~{\rm \eqref{musigma}}. Then, the isotropic random field 
\begin{equation}\label{carfield} S_\lambda(\bm{t}) := -\int\limits_{\mathbb{R}^d} \frac{1}{2\lambda} e^{-\lambda||\bm{t}-\bm{u}||}d{L(\bm{u})}, \ \lambda>0, \ \bf{t}\in\mathbb{R}^d,\end{equation} is called CAR(1) field. Its mean equals
\[ ES_\lambda(\bm{t}) = -\frac{\mu \pi^{d/2}\Gamma(d+1)}{2\lambda|\lambda|^d\Gamma(d/2+1)},\] the spectral density of $S_\lambda(\cdot)$ is
\begin{equation}\label{speccar} f(\bm{\omega}): = \frac{c_1^2 \sigma^2}{(||\bm{\omega}||^2+\lambda^2)^{d+1}} ,\ \bm{\omega}\in\mathbb{R}^d, \end{equation} and the covariance function of $S_\lambda(\cdot)$ is given by
\begin{equation} \label{mat_cov}\gamma(\bm{t}): = c_1^2 \sigma^2 \left( \frac{\pi}{2} \right)^{d/2} \frac{||\lambda \bm{t}||^{d/2+1}}{|\lambda|^{d+2} \Gamma(d+1)} K_{d/2+1}(||\lambda\bm{t}||),  \end{equation} where
\begin{equation*} c_1:=
    \begin{cases}
      -2^{d/2-1}\Gamma\left( \frac{d+1}{2} \right)/\sqrt{\pi},\ \textrm{if $d$ is odd}, \\
     -2^{-d/2}\Gamma(d)/\Gamma\left( \frac{d}{2} \right), \ \textrm{if $d$ is even},
    \end{cases}\
\end{equation*} and $K_{d/2+1}(\cdot)$ denotes the modified Bessel function of the second kind of order $d/2+1.$
\end{theorem}

\begin{remark}
The integral in {\rm \eqref{carfield}} is interpreted as the Rajput-Rosinski integral with respect to the random measure $\Lambda_0(\cdot)$ corresponding to the L\'evy sheet $L(\cdot)$ from \eqref{Levyfield}. The notation $dL(\bm{u})$ is due to Koshnevisan and Nualart {\rm \cite{Nualart}}.
\end{remark}

\begin{remark}
The covariance function in {\rm \eqref{mat_cov}} belongs to the Mat\'ern class, which has a general form, see~{\rm \cite{stein2012interpolation}},
\[ \gamma_\nu(\bm{t}) :=  \sigma^2 ||a\bm{t}||^{\nu} K_\nu(||a\bm{t}||), \ a>0, \ \nu>0.\] As the function $K_\nu(x)$ decays exponentially when $x\to\infty,$ {\rm \cite[10.40.2]{DLMF}}, the Mat\'ern covariance is also exponentially bounded for large values of its argument.
\end{remark}

\section{Superpositions of continuous autoregressive fields} \label{sec:supcar}
In this section, we introduce superpositions of continuous autoregressive fields (supCAR fields), and study conditions for their existence, their marginal, finite-dimensional distributions, and dependence structures.

Consider a random field
\begin{equation}\label{supcar} X(\bm{t}) := -\int\limits_0^\infty\int\limits_{\mathbb{R}^d} \frac{1}{2\lambda} e^{-\lambda||\bm{t}-\bm{u}||}\Lambda(d\bm{u}, d\lambda),\ \bm{t}\in\mathbb{R}^d,\end{equation} where $\Lambda(\cdot)$ is an independently scattered scalar infinitely divisible random measure such that  for $A\in\mathcal{S}'$ its cumulant function  allows the following factorisation
\begin{equation}
\label{cum_supcar}
C(s \ddagger \Lambda(A)) = [\pi\times Leb](A) \cdot \mathcal{K}(s),\end{equation} where $\pi(\cdot)$ is a probability measure on $\mathbb{R}^+,$ and $\mathcal{K}(\cdot)$ is a cumulant function of some infinitely divisible distribution. In this paper, without loss of generality, we restrict our attention to the case when $a=0$ such that 
\begin{equation}\label{cum_supcar1}
\mathcal{K}(s) = - \frac{b}{2}s^2 + \int\limits_\mathbb{R}(e^{is x} - 1 - is\tau(x))W(dx).
\end{equation} The characteristic quadruple of the supCAR field \eqref{supcar} is $(0,b,W,\pi).$

\begin{remark}
For a fixed $\lambda,$ the inner integral in {\rm \eqref{supcar}} is a CAR(1) field $S_\lambda(\cdot)$. Thus, one can interpret the field {\rm \eqref{supcar}} as an infinite superposition of CAR(1) fields $S_\lambda(\cdot)$.
\end{remark}

The following result provides conditions for the existence of supCAR random fields. 

\begin{theorem}\label{th:ex1}
The supCAR random field with the integral representation {\rm \eqref{supcar}} is well-defined in the following cases:
\begin{itemize}
\item[{\rm(i)}] when $b>0,$  
\begin{equation}\label{bneq} \int\limits_0^\infty\frac{|\ln (\lambda)|^{d}}{\lambda^{d+2}}\pi(d\lambda)<\infty {\mbox \ \ and \ } \int\limits_{1\leq|x|<\infty}|x|(\ln|x|)^{d-1}W(dx)<\infty;\end{equation}
\item[{\rm(ii)}] when $b=0,$ 
\[  \int\limits_0^\infty\frac{|\ln (\lambda)|^{d-1}}{\lambda^{d+1}}\pi(d\lambda)<\infty, \ \ \ \int\limits_{1\leq|x|<\infty}|x|(\ln|x|)^{d-1}W(dx)<\infty,\] 
\begin{equation}\label{beq}
{}\hspace{-1cm}\int\limits_0^{1/2}\frac{|\ln (\lambda)|^{d-1}}{\lambda^{d+2}}\int\limits_{-2\lambda}^{2\lambda} x^2 W(dx) \pi(d\lambda) <\infty\ {\mbox and} \  \int\limits_0^{1/2}\frac{|\ln (\lambda)|^{d}}{\lambda^{d+2-k}}\int\limits_{2\lambda\leq |x|\leq1} |x|^{2-k} W(dx) \pi(d\lambda) <\infty,\ 
\end{equation} 
\end{itemize} for $k=1,2.$

\end{theorem}

\begin{remark}
As will be shown, the conditions in {\rm \eqref{bneq}} are more restrictive than the conditions~{\rm \eqref{beq}}. The reason is that the parameter $b$ in \eqref{cum_supcar1} corresponds to the variances of the Gaussian components of the underlying CAR(1) fields $S_\lambda(\cdot)$ used in the superposition~\eqref{supcar}. If $b>0,$ controlling the finiteness of the variance of the Gaussian component of the supCAR field requires the additional condition.
\end{remark}

\begin{proof}
For the case (ii), the integral representation \eqref{supcar} is well-defined if, consult \cite[Proposition 2.7]{rajput1989spectral},
\begin{equation} \label{I1}
I_0:=\int\limits_0^\infty \int\limits_{\mathbb{R}^d} V_0\left(-\frac{1}{2\lambda}e^{-\lambda||\bm{y}||}\right)d\bm{y}\pi(d\lambda)<\infty
\end{equation} and
\begin{equation} \label{I2}
I_1:=\int\limits_0^\infty \int\limits_{\mathbb{R}^d} \left|V_1\left(-\frac{1}{2\lambda}e^{-\lambda||\bm{y}||}\right)\right|d\bm{y}\pi(d\lambda)<\infty, 
\end{equation} where $V_0(r):=\int_\mathbb{R}\min(1, (rx)^2)W(dx),$ $V_1(r):=\int_\mathbb{R}(\tau(rx)-r\tau(x))W(dx),$ $r\ge 0,$ and $\tau(\cdot)$ is defined in~\eqref{tau}. 

First, let us consider the integral in \eqref{I1}. Using the $d$-dimensional spherical coordinates and the change of variables $e^{-\lambda\rho}/(2\lambda)=r$, one gets 
\[I_0 = C\int\limits_0^\infty \int\limits_0^\infty\rho^{d-1}V_0\left(\frac{1}{2\lambda } e^{-\lambda\rho} \right) d\rho \pi(d\lambda)= C(-1)^{d+1} \int\limits_0^\infty \frac{1}{\lambda^d} \int\limits_0^{\frac{1}{2\lambda}} r^{-1}\left(\ln (2\lambda r)\right)^{d-1} V_0(r)dr \pi(d\lambda)\]
\[\leq C\int\limits_0^\infty \frac{1}{\lambda^d} \int\limits_0^{\frac{1}{2\lambda}} r^{-1}|\ln (2r^2)|^{d-1}V_0(r)dr \pi(d\lambda)+C\int\limits_0^\infty \frac{|\ln (2\lambda^2)|^{d-1}}{\lambda^d} \int\limits_0^{\frac{1}{2\lambda}} r^{-1}V_0(r)dr\pi(d\lambda)\]\[:=C(I_{0,1}+I_{0,2}).\] The integral $I_{0,1}$ can be represented as
\[I_{0,1} = C\int\limits_0^\infty \frac{1}{\lambda^d} \int\limits_0^{\frac{1}{2\lambda}} r|\ln (2r^2)|^{d-1}\int\limits_{|x|<r^{-1}}x^2 W(dx) dr \pi(d\lambda) \] \[ + C\int\limits_0^\infty \frac{1}{\lambda^d} \int\limits_0^{\frac{1}{2\lambda}} r^{-1}|\ln (2r^2)|^{d-1}\int\limits_{r^{-1}\leq |x|}W(dx) dr \pi(d\lambda):=C(I_{0,1,1}+I_{0,1,2}).\] One can estimate the integral $I_{0,1,1}$ as
\begin{eqnarray} I_{0,1,1} &\leq & C\Bigg(\int\limits_0^\frac{1}{2} \frac{1}{\lambda^d} \int\limits_0^{1} r|\ln (2r^2)|^{d-1}\int\limits_{|x|<r^{-1}}x^2 W(dx) dr \pi(d\lambda)\nonumber\\ & +& \int\limits_0^\frac{1}{2} \frac{|\ln (2\lambda^2)|^{d-1}}{\lambda^d} \int\limits_1^{\frac{1}{2\lambda}} r\int\limits_{|x|<r^{-1}}x^2 W(dx) dr \pi(d\lambda)\nonumber\\  & + &\int\limits_\frac{1}{2}^\infty \frac{1}{\lambda^d} \int\limits_0^{\frac{1}{2\lambda}} r|\ln (2r^2)|^{d-1}\int\limits_{|x|<r^{-1}}x^2 W(dx) dr \pi(d\lambda) \Bigg) := C(I_{0,1,1}^{(1)} + I_{0,1,1}^{(2)} + I_{0,1,1}^{(3)}).\nonumber\end{eqnarray}  The integral in $I_{0,1,1}^{(1)}$ can be represented as
\begin{equation}\label{Th2:tmp} \int\limits_{-1}^1x^2\int\limits_0^1r|\ln(2r^2)|^{d-1}drW(dx) + \int\limits_{1\leq |x| <\infty}x^2\int\limits_0^{\frac{1}{|x|}}r|\ln(2r^2)|^{d-1}drW(dx).\end{equation} The first integral in \eqref{Th2:tmp} is finite as $\int_{\mathbb{R}}\min(1,x^2)W(dx)<\infty.$ 

The change of variables $r=e^{-t/2}$ results in
\begin{equation}\label{th2:tmp2}
\int\limits_0^{\frac{1}{|x|}}r|\ln(2r^2)|^{d-1}dr\leq C\int\limits_0^{\frac{1}{|x|}}r dr + C\int\limits_0^{\frac{1}{|x|}}r|\ln(r)|^{d-1}dr = \frac{C}{|x|^2} + C\int\limits\limits_{\ln x^2}^\infty t^{d-1} e^{-t},\end{equation} Note that
\[ \int\limits_{\ln x^2}^\infty t^{d-1} e^{-t} dt = \Gamma(d,\ln x^2),\] where the function $\Gamma(\cdot,\cdot)$ is the incomplete Gamma function. As $\Gamma(d,\ln x^2)\sim \frac{1}{x^2}\left( \ln x^2 \right)^{d-1},$ by \eqref{th2:tmp2} the second integral in \eqref{Th2:tmp} is finite if
\[ \int\limits_{1\leq |x| <\infty} (\ln x^2)^{d-1} W(dx)<\infty.\] Therefore, \eqref{beq} gives us the finiteness of the integrals $I_{0,1,1}^{(1)}$. The same computations are valid for $I_{0,1,1}^{(3)}.$ Now, let us consider the integral $I_{0,1,1}^{(2)}$
\[ I_{0,1,1}^{(2)} = \int\limits_{0}^{\frac{1}{2}}\frac{|\ln(2\lambda^2)|^{d-1}}{\lambda^d} \int\limits_{-2\lambda}^{2\lambda}x^2\int\limits_{1}^{\frac{1}{2\lambda}} r drW(dx)\pi(d\lambda) + \int\limits_{0}^{\frac{1}{2}}\frac{|\ln(2\lambda^2)|^{d-1}}{\lambda^d} \int\limits_{2\lambda\leq|x|\leq1}x^2\int\limits_{1}^{\frac{1}{|x|}} r drW(dx)\pi(d\lambda)\]
\[\leq C \int\limits_{0}^{\frac{1}{2}} \frac{\left|\ln\left( 2\lambda^2 \right)\right|^{d-1}}{\lambda^{d+2}}\int\limits_{-2\lambda}^{2\lambda} x^2 W(dx)\pi(d\lambda)+\int\limits_{0}^{\frac{1}{2}} \frac{\left|\ln\left( 2\lambda^2 \right)\right|^{d-1}}{\lambda^d}\int\limits_{2\lambda\leq |x| \leq 1} W(dx)\pi(d\lambda).\] Thus, by \eqref{beq} the integral $I_{0,1,1}^{(2)}$ and, therefore, the integral $I_{0,1,1}$ are finite.

Let us study $I_{0,1,2}$
\begin{eqnarray} I_{0,1,2} & = &\int\limits_0^\frac{1}{2} \frac{1}{\lambda^d} \int\limits_0^{1} r^{-1}|\ln (2r^2)|^{d-1}\int\limits_{r^{-1} \leq |x|} W(dx) dr \pi(d\lambda) \nonumber\\ &+& \int\limits_0^\frac{1}{2} \frac{|\ln (2\lambda^2)|^{d-1}}{\lambda^d} \int\limits_1^{\frac{1}{2\lambda}} r^{-1}\int\limits_{r^{-1} \leq |x|} W(dx) dr \pi(d\lambda) \nonumber\\
& + &\int\limits_\frac{1}{2}^\infty \frac{1}{\lambda^d} \int\limits_0^{\frac{1}{2\lambda}} r^{-1}|\ln (2r^2)|^{d-1}\int\limits_{r^{-1} \leq |x|} W(dx) dr \pi(d\lambda) := I_{0,1,2}^{(1)} + I_{0,1,2}^{(2)} + I_{0,1,2}^{(3)}.\nonumber\end{eqnarray}  Note that
\[\int\limits_{0}^1r^{-1}|\ln (2r^2)|^{d-1}\int\limits_{r^{-1}\leq |x|} W(dx) dr = \int\limits_{1\leq|x|<\infty}\int\limits_{\frac{1}{|x|}}^1r^{-1}|\ln(2r^2)|^{d-1}drW(dx)\] 
\begin{equation} \leq \int\limits_{1 \leq |x| <\infty}|x|\left(\ln2+2|\ln x|\right)^{d-1}\left(1 - \frac{1}{|x|}\right)W(dx).\end{equation} The last integral is finite due to \eqref{beq}. Thus, the integrals $I_{0,1,2}^{(1)}$ and $I_{0,1,2}^{(3)}$ are finite.

Let us estimate $I_{0,1,2}^{(2)},$ it equals 
\[\int\limits_0^\frac{1}{2} \frac{|\ln (2\lambda^2)|^{d-1}}{\lambda^d}\int\limits_{2\lambda\leq |x| \leq 1} \int\limits_{\frac{1}{|x|}}^{\frac{1}{2\lambda}} r^{-1} dr W(dx)\pi(d\lambda) + \int\limits_0^\frac{1}{2} \frac{|\ln (2\lambda^2)|^{d-1}}{\lambda^d}\int\limits_{1\leq |x| < \infty} \int\limits_{1}^{\frac{1}{2\lambda}} r^{-1} dr W(dx) \pi(d\lambda)\]
\[\leq C\int\limits_0^\frac{1}{2} \frac{|\ln (2\lambda^2)|^{d}}{\lambda^d}\int\limits_{2\lambda\leq |x| \leq 1}W(dx)\pi(d\lambda) + C\int\limits_0^\frac{1}{2} \frac{|\ln (2\lambda^2)|^{d}}{\lambda^d}\pi(d\lambda),\] where the last bound follows from $\int\limits_{1\leq |x| < \infty} W(dx)<\infty.$ Thus, $I_{0,1,2}^{(2)}$ is finite by~\eqref{beq} and the integral $I_{0,1}$ is finite too.

The integral $I_{0,2}$ is equal
\[\int\limits_0^\infty \frac{|\ln (2\lambda^2)|^{d-1}}{\lambda^d} \int\limits_0^{\frac{1}{2\lambda}} r\int\limits_{|x|\leq r^{-1}}x^2 W(dx) dr \pi(d\lambda) + \int\limits_0^\infty \frac{|\ln (2\lambda^2)|^{d-1}}{\lambda^d} \int\limits_0^{\frac{1}{2\lambda}} r^{-1}\int\limits_{r^{-1}\leq|x|}W(dx) dr \pi(d\lambda)\]\[:=I_{0,2,1}+I_{0,2,2}.\] Rewrite the integral $I_{0,2,1}$ as 
\[ \int\limits_0^\frac{1}{2} \frac{|\ln (2\lambda^2)|^{d-1}}{\lambda^d} \int\limits_0^{1} r\int\limits_{|x|\leq r^{-1}}x^2 W(dx) dr \pi(d\lambda) + \int\limits_0^\frac{1}{2} \frac{|\ln (2\lambda^2)|^{d-1}}{\lambda^d} \int\limits_1^{\frac{1}{2\lambda}} r\int\limits_{|x|\leq r^{-1}}x^2 W(dx) dr \pi(d\lambda) \]
\[ + \int\limits_\frac{1}{2}^\infty \frac{|\ln (2\lambda^2)|^{d-1}}{\lambda^d} \int\limits_0^{\frac{1}{2\lambda}} r\int\limits_{|x|\leq r^{-1}}x^2 W(dx) dr \pi(d\lambda) := I_{0,2,1}^{(1)} + I_{0,2,1}^{(2)} + I_{0,2,1}^{(3)}.\] The following integral is finite as
\[ \int\limits_0^1r\int\limits_{|x|<r^{-1}}x^2W(dx)dr = \int\limits_{-1}^1x^2\int\limits_{0}^1 rdrW(dx) + \int\limits_{1\leq|x|<\infty} x^2 \int\limits_{0}^\frac{1}{|x|} r dr W(dx) < \infty.\] Thus, by \eqref{beq} the integrals $I_{0,2,1}^{(1)}$ and $I_{0,2,1}^{(3)}$ are finite. The integral $I_{0,2,1}^{(2)}$ is identical to the integral $I_{0,1,1}^{(2)}.$ Thus, the integral $I_{0,2,1}$ is finite.

Now, let us consider $I_{0,2,2}.$ One can represent it as
\[ \int\limits_0^\frac{1}{2} \frac{|\ln (2\lambda^2)|^{d-1}}{\lambda^d} \int\limits_0^{1} r^{-1}\int\limits_{r^{-1} \leq |x|} W(dx) dr \pi(d\lambda) + \int\limits_0^\frac{1}{2} \frac{|\ln (2\lambda^2)|^{d-1}}{\lambda^d} \int\limits_1^{\frac{1}{2\lambda}} r^{-1}\int\limits_{r^{-1} \leq |x|} W(dx) dr \pi(d\lambda) \]
\[ + \int\limits_\frac{1}{2}^\infty \frac{|\ln (2\lambda^2)|^{d-1}}{\lambda^d} \int\limits_0^{\frac{1}{2\lambda}} r^{-1}\int\limits_{r^{-1} \leq |x|} W(dx) dr \pi(d\lambda) := I_{0,2,2}^{(1)} + I_{0,2,2}^{(2)} + I_{0,2,2}^{(3)}.\] As 
\[ \int\limits_0^1r^{-1}\int\limits_{r^{-1}\leq|x|}W(dx)dr = \int\limits_{1\leq |x| <\infty} \int\limits_{\frac{1}{|x|}}^1 r^{-1} dr W(dx) = \int\limits_{1\leq |x| <\infty} \ln|x| W(dx) < \infty,\] by condition \eqref{beq}, the integrals $I_{0,2,2}^{(1)}$ and $I_{0,2,2}^{(3)}$ are finite. Because the integral $I_{0,2,2}^{(2)}$ is identical to the integral $I_{0,1,2}^{(2)},$ the integral $I_{0,2}$ is finite too, which verifies the finiteness of $I_0$ in \eqref{I1}.

Now, let us show the boundedness of the integral \eqref{I2}. Using the $d$-dimensional spherical coordinates and the change of variables $r=e^{-\lambda\rho}/(2\lambda)$, one gets 
\[I_1 \leq C\int\limits_0^\infty \int\limits_0^\infty\rho^{d-1}\left|V_1\left(\frac{1}{2\lambda } e^{-\lambda\rho} \right)\right| d\rho \pi(d\lambda) = C(-1)^d \int\limits_0^\infty \frac{1}{\lambda^d} \int\limits_0^{\frac{1}{2\lambda}} r^{-1}\left(\ln (2\lambda r)\right)^{d-1} |V_1(r)|dr \pi(d\lambda)\] \[\leq C\int\limits_0^\infty \frac{1}{\lambda^d} \int\limits_0^{\frac{1}{2\lambda}} r^{-1}|\ln (2r^2)|^{d-1}|V_1(r)|dr \pi(d\lambda)+C\int\limits_0^\infty \frac{|\ln (2\lambda^2)|^{d-1}}{\lambda^d} \int\limits_0^{\frac{1}{2\lambda}} r^{-1}|V_1(r)|dr\pi(d\lambda)\]\[:=C(I_{1,1}+I_{1,2}).\] The integral $I_{1,1}$ can be represented as the next sum of the integrals
\[I_{1,1} = \int\limits_0^\frac{1}{2} \frac{1}{\lambda^d} \int\limits_0^{1} r^{-1}|\ln (2r^2)|^{d-1}|V_1(r)| dr \pi(d\lambda) + \int\limits_0^\frac{1}{2} \frac{|\ln (2\lambda^2)|^{d-1}}{\lambda^d} \int\limits_1^{\frac{1}{2\lambda}} r^{-1} |V_1(r)| dr \pi(d\lambda) \]
\[ + \int\limits_\frac{1}{2}^\infty \frac{1}{\lambda^d} \int\limits_0^{\frac{1}{2\lambda}} r^{-1}|\ln (2r^2)|^{d-1} |V_1(r)| dr \pi(d\lambda) := I_{1,1,1} + I_{1,1,2} + I_{1,1,3}.\] The function $V_1(\cdot)$ can be represented as 
$V_1(r)=\sum_{i=1}^5V_1^{(i)}(r),$ see \cite{barndorff2001superposition}, where
\[ V_1^{(1)}(r) := r\int\limits_{1 < x \leq r^{-1}}(x-1)W(dx),\ V_1^{(2)}(r) := r\int\limits_{-r^{-1} < x \leq -1}(x+1)W(dx),\]
\[ V_1^{(3)}(r) := \int\limits_{r^{-1} < x \leq 1,}(1-rx)W(dx),\  V_1^{(4)}(r) := \int\limits_{-1 \leq x  < -r^{-1}}(-1-rx)W(dx),\]
\[ V_1^{(5)}(r) := (1-r)\int\limits_{\max\{1,r^{-1}\} < |x|}\sign(x)W(dx),\] and it is equal zero in other cases of the integration regions, or if the upper limits in the above integrals are smaller that the lower ones.

Therefore, the integrals $I_{1,1,1}$ and $I_{1,1,3}$ are finite if 
\begin{equation*} \int\limits_{0}^\infty \frac{\pi(d\lambda)}{\lambda^d} <\infty \ {\rm and} \  \int\limits_{0}^1 r^{-1}|\ln(2r^2)|^{d-1}|V_1^{(i)}(r)|dr <\infty, \ i=\overline{1,5}.\end{equation*} By \eqref{beq}, the first integral above is finite. Let us show that the second integral is finite too. First, for $i=1$ and $2,$ we get 
\[ \int\limits_{0}^1 r^{-1}|\ln(2r^2)|^{d-1}|V_1^{(i)}(r)|dr \leq \int\limits_{0}^1 |\ln(2r^2)|^{d-1} \int\limits_{1}^{r^{-1}}(x-1)W(dx)dr\] \[\leq \int\limits_0^1|\ln(2r^2)|^{d-1}dr\int\limits_{1}^\infty x W(dx)<+\infty.\] The first integral in the product above is finite as $|\ln(2r^2)|^{d-1}<r^{-\frac{1}{2}}$ in some neighbourhood of the origin, where the second integral is finite due to condition \eqref{beq}. 

The integrals involving $V_1^{(3)}(\cdot)$ and $V_1^{(4)}(\cdot)$ equal to $0,$ as $r^{-1}>1.$ For the case of $V_1^{(5)}(\cdot),$ we proceed as follows
\[ \int\limits_{0}^1 r^{-1}|\ln(2r^2)|^{d-1}|V_1^{(5)}(r)|dr \leq \int\limits_0^1 r^{-1}(1-r)|\ln(2r^2)|^{d-1}\int\limits_{r^{-1}<|x|}W(dx) dr \] \[= \int\limits_{1\leq |x| <\infty} \int\limits_{\frac{1}{|x|}}^1 r^{-1}(1-r)|\ln(2r^2)|^{d-1} dr W(dx) \leq C_1 \int\limits_{1\leq |x| <\infty} |x|(C+\ln|x|)^{d-1}W(dx) < \infty,\] which is due to \eqref{beq}. Thus, the integrals $I_{1,1,1}$ and $I_{1,1,3}$ are finite. 

Let us investigate the integral $I_{1,1,2}.$ To show its finiteness we will prove that 
\begin{equation}\label{th:ex_int_K} \int\limits_1^{\frac{1}{2\lambda}} r^{-1}|V_1^{(i)}(r)| dr<\infty, \ i=\overline{1,5}.\end{equation}
For the cases $i=1$ and $2,$ as $r^{-1}>1$ in \eqref{th:ex_int_K}, $V_1^{(1)}(r)$ and the corresponding integrals vanish. For the cases $i=3$ and $4,$ the following estimates hold true
\[ \int\limits_{1}^{\frac{1}{2\lambda}} r^{-1}|V_1^{(i)}(r)|dr \leq \int\limits_{1}^{\frac{1}{2\lambda}} r^{-1}\int\limits_{r^{-1} < x\leq 1} W(dx) dr + \int\limits_{1}^{\frac{1}{2\lambda}}\int\limits_{r^{-1} < x\leq 1} x W(dx) dr \]
\[ = \int\limits_{2\lambda}^{1} \int\limits_{\frac{1}{x}}^{\frac{1}{2\lambda}} r^{-1}dr W(dx) + \int\limits_{2\lambda}^{1} x \int\limits_{\frac{1}{x}}^{\frac{1}{2\lambda}}dr W(dx) \leq 2|\ln(2\lambda)|\int\limits_{2\lambda}^1W(dx) + \frac{1}{2\lambda}\int\limits_{2\lambda}^1 xW(dx) + \int\limits_{2\lambda}^1W(dx).\] For the case when $i=5,$ we obtain
\[ \int\limits_{1}^{\frac{1}{2\lambda}} r^{-1}|V_1^{(5)}(r)|dr \leq \int\limits_{1}^{\frac{1}{2\lambda}} r^{-1}|1-r| \int\limits_{\max\{1,r^{-1}\}<|x| } W(dx)dr\leq\frac{C}{\lambda},\] as $\int_{1\leq |x| <\infty}W(dx)<\infty.$ By the obtained estimates and  \eqref{beq} the integral $I_{1,1}$ is finite.

Let us investigate the integral $I_{1,2}$
\[I_{1,2}=\int\limits_0^\frac{1}{2} \frac{|\ln (2\lambda^2)|^{d-1}}{\lambda^d} \int\limits_0^{1} r^{-1}|V_1(r)| dr \pi(d\lambda) + \int\limits_0^\frac{1}{2} \frac{|\ln (2\lambda^2)|^{d-1}}{\lambda^d} \int\limits_1^{\frac{1}{2\lambda}} r^{-1} |V_1(r)| dr \pi(d\lambda) \]
\[ + \int\limits_\frac{1}{2}^\infty \frac{|\ln (2\lambda^2)|^{d-1}}{\lambda^d} \int\limits_0^{\frac{1}{2\lambda}} r^{-1} |V_1(r)| dr \pi(d\lambda) := I_{1,2,1} + I_{1,2,2} + I_{1,2,3}.\] The integrals $I_{1,2,1}$ and $I_{1,2,3}$ are finite if 
\[ \int\limits_{0}^\infty \frac{|\ln(2\lambda^2)|^{d-1}}{\lambda^d}\pi(d\lambda) <\infty, \ {\rm and} \  \int\limits_{0}^1 r^{-1}|V_1^{(i)}(r)|dr <\infty, \ i=\overline{1,5}.\] The second  integrals involving $V_1^{(i)}(\cdot)$ are finite because, as was shown before, for a sufficiently small $\delta>0,$ they are bounded by the finite integrals 
\[\int\limits_{0}^\delta r^{-1}|V_1^{(i)}(r)|dr\leq C \int\limits_{0}^\delta r^{-1}|\ln(2r^2)|^{d-1}|V_1^{(i)}(r)|dr<\infty, \ i=\overline{1,5}.\]
As the integral $I_{1,2,2}$ equals to the integral $I_{1,1,2},$ the integral $I_{1,2}$ is finite. This completes the proof for the case {\rm (ii)}.

For the case {\rm (i)} of $b>0,$ the integral \eqref{supcar} is well-defined \cite[Proposition 2.7]{rajput1989spectral} if the conditions \eqref{I1} and \eqref{I2} are satisfied and 
\begin{equation}\label{I3} \int\limits_0^\infty \int\limits_{\mathbb{R}^d} \frac{1}{\lambda^2}e^{-2\lambda||\bm{y}||}ds\pi(d\lambda)<\infty.\end{equation}
By using $d$-dimensional spherical coordinates, one gets that the integral in \eqref{I3} is bounded by
\[C \int\limits_0^\infty \int\limits_{0}^\infty \frac{r^{d-1}}{\lambda^2} e^{-2\lambda r} dr \pi(d\lambda) =C\int\limits_0^\infty \frac{1}{\lambda^{d+2}}\left(\int\limits_0^\infty r^{d-1} e^{-2r}dr\right)\pi(d\lambda) \leq C\int\limits_0^\infty\frac{\pi(d\lambda)}{\lambda^{d+2}}.\]
Thus, the integral \eqref{I3} is finite if the last integral above is finite, which is true due to the finiteness of the first integral in \eqref{bneq}.

Let us show that the conditions \eqref{I1} and \eqref{I2} are satisfied if the assumptions in \eqref{bneq} hold true. The assumptions in \eqref{bneq} imply the finiteness of the first two integrals in the case {\rm (ii)}. Thus, it is enough to show the finiteness of the remaining integrals in \eqref{beq}.

It follows from \eqref{eq:W_int} that for $\lambda\in[0,\frac{1}{2}]$ it holds $\int\limits_{-2\lambda}^{2\lambda}x^2 W(dx)<+\infty.$ Hence, by the first assumption in \eqref{bneq}, the first integral in \eqref{beq} is finite. 

Now, note that by \eqref{eq:W_int} for $\lambda\in[0,\frac{1}{2}]$ we have $\int\limits_{2\lambda\leq |x| \leq 1}x^2W(dx)<+\infty.$ Thus,  
\begin{equation*} \int\limits_0^{\frac{1}{2}} \frac{|\ln(\lambda)|^d}{\lambda^{d+2-k}}\int\limits_{2\lambda<|x|\leq1} |x|^{2-k}W(dx) \pi(d\lambda)\leq C \int\limits_0^{\frac{1}{2}} \frac{|\ln(\lambda)|^{d}}{\lambda^{d+2}}\int\limits_{2\lambda < |x| \leq 1}x^2W(dx)\pi(d\lambda)<+\infty. \end{equation*} \end{proof}

As follows from the proof of Theorem \ref{th:ex1}, it holds true under the following slightly more restrictive conditions for the case $\rm (ii).$ These conditions are easier to verify as they deal with each of the measures $\pi(\cdot)$ and $W(\cdot)$ separately.

\begin{corollary}
The supCAR random fields are well-defined if 
\[\int\limits_{0}^{\infty}\frac{|\ln(\lambda)|^d}{\lambda^{d+2}}\pi(d\lambda)<+\infty \ \ {\rm and} \ \ \int\limits_{1\leq |x| \leq \infty}|x|(\ln(|x|))^{d-1}W(dx)<+\infty.\]
\end{corollary}

\begin{example}\label{ex:inverse_gauss}
Let $\pi(\cdot)$ be the Gamma measure $\Gamma(H,1), \ H>0,$ with the density 
\begin{equation}\label{gamma:density} p(\lambda) = \frac{\lambda^{H-1}e^{-\lambda}}{\Gamma(H)}\, \mathbb{1}_{(0,\infty)}(\lambda), \end{equation} and $W(\cdot)$ be the inverse Gaussian L\'evy measure with the density
\begin{equation}\label{eq:inv_gauss} \frac{dW(x)}{dx} = \sqrt{\frac{\alpha}{2\pi x^3}} \exp\left(-\frac{\alpha(x-\mu)^2}{2\mu^2x}  \right) \mathbb{1}_{(0,\infty)}(x) , \ \alpha,\mu>0.\end{equation} It is easy to see that if the Gaussian component is present, i.e. $b>0$, the integrals in the conditions \eqref{bneq} are finite if $H>d+2.$ If the Gaussian component is absent, i.e. $b=0$, the conditions in \eqref{beq} are satisfied if $H>d+1.$

The Blumental-Getoor index of the inverse Gaussian L\'evy measure \eqref{eq:inv_gauss} is $\beta_{BG}=0.$
\end{example}

In Section \ref{sect:limit}, to obtain limit theorems in a general form, we assume that the probability measure $\pi(\cdot)$ has a density that varies regularly in a neighbourhood of the origin. The corresponding existence conditions for the supCAR field \eqref{supcar} are as follows.
\begin{proposition}\label{prop:elambdast}
Let the probability measure $\pi(\cdot)$ have a density $p(\lambda)\sim l(1/\lambda)\lambda^{\alpha-1},$ when $\lambda\to0,$ where $l(\cdot)$ is a slowly varying at infinity function, and $W(\cdot)$ be a L\'evy measure with the Blumental-Getoor index $\beta_{BG}.$ Then, the integral representation  {\rm \eqref{supcar}} of the supCAR field is well-defined if  
\[\int\limits_{1\leq|x|<\infty}|x|(\ln|x|)^{d-1}W(dx)<\infty\] and one of the following conditions hold true:
\begin{itemize}
\item[\rm{(i)}] $b>0$ and $\alpha>d+2,$ or
\item[\rm{(ii)}] $b=0$ and $\alpha>\max\left(d+\beta_{BG}, d+1\right).$
\end{itemize}
\end{proposition} 

\begin{proof} Let us check the conditions {\rm (ii)} of Theorem \ref{th:ex1}. For any $\delta>0,$ it holds
\[\int\limits_{0}^\infty\frac{|\ln(\lambda)|^{d-1}}{\lambda^{d+1}}\pi(d\lambda) = \int\limits_{0}^\delta \frac{|\ln(\lambda)|^{d-1}}{\lambda^{d+1}} \pi(d\lambda)+\int\limits_\delta^\infty \frac{|\ln(\lambda)|^{d-1}}{\lambda^{d+1}} \pi(d\lambda).\]
As $\pi(\cdot)$ is a probability measure, the second integral in the above expression is, obviously, finite. As $l(\cdot)$ is a slowly varying function, for every $\varepsilon>0,$ it holds $l(1/\lambda)\leq 1/\lambda^{\varepsilon},$ when $\lambda\to0.$ Thus, one can bound the first integral  as
\[\int\limits_{0}^\delta \frac{|\ln(\lambda)|^{d-1}}{\lambda^{d+1}} \pi(d\lambda) \leq C \int\limits_0^{\delta}\frac{|\ln(\lambda)|^{d-1}l(1/\lambda)}{\lambda^{d+2-\alpha}}d\lambda \leq C \int\limits_0^{\delta}\frac{d\lambda}{\lambda^{d+2-\alpha+2\varepsilon}}.\] The upper bound is finite if $\alpha>d+1$ and $\varepsilon$ is selected as $2\varepsilon<\alpha-d-1.$

Now, let us show the finiteness of the first integral in \eqref{beq}. It can be estimated as 
\[\int\limits_0^{1/2}\frac{|\ln (\lambda)|^{d-1}}{\lambda^{d+2}}\int\limits_{-2\lambda}^{2\lambda} x^{2-\beta_{BG}+\beta_{BG}} W(dx) \pi(d\lambda)\leq \int\limits_0^{1/2}\frac{|\ln (\lambda)|^{d-1}}{\lambda^{d+\beta_{BG}}}\int\limits_{-2\lambda}^{2\lambda} x^{\beta_{BG}} W(dx) \pi(d\lambda) \] \[
\leq C \int\limits_0^{1/2}\frac{|\ln (\lambda)|^{d-1}l\left(1/\lambda\right)}{\lambda^{d+1+\beta_{BG}-\alpha}}\int\limits_{-2\lambda}^{2\lambda} x^{\beta_{BG}} W(dx) d\lambda \leq C\int\limits_0^{1/2}\frac{d\lambda}{\lambda^{d+1+\beta_{BG} -\alpha+2\varepsilon}},\]
where the last integral is finite if $\alpha>d+\beta_{BG}$ as $\varepsilon$ can be chosen sufficiently small.

Let us proceed with the second integral in \eqref{beq}. If $\beta_{BG}\in(0,1],$ and $k=1,$ it holds
\[ \int\limits_0^{1/2}\frac{|\ln (\lambda)|^{d}}{\lambda^{d+1}}\int\limits_{2\lambda\leq |x|\leq1} |x| W(dx)\pi(d\lambda)\leq \int\limits_0^{1/2}\frac{|\ln (\lambda)|^{d}}{\lambda^{d+1}}\int\limits_{2\lambda\leq |x|\leq1} |x|^{\beta_{BG}} W(dx)\pi(d\lambda)\]\[ \leq C \int\limits_0^{1/2}\frac{1}{\lambda^{d+2-\alpha+2\varepsilon}}d\lambda,\] where the last integral is finite if $\alpha>d+1$. For $\beta_{BG}\in(0,1],$ and $k=2,$ it holds 
\[ \int\limits_0^{1/2}\frac{|\ln (\lambda)|^{d}}{\lambda^{d}}\int\limits_{2\lambda\leq |x|\leq1} W(dx)\pi(d\lambda) \leq \int\limits_0^{1/2}\frac{|\ln (\lambda)|^{d}}{\lambda^{d+\beta_{BG}}}\int\limits_{2\lambda\leq |x|\leq1}|x|^{\beta_{BG}} W(dx)\pi(d\lambda)\]\[ \leq C \int\limits_0^{1/2}\frac{1}{\lambda^{d+1+\beta_{BG}-\alpha+2\varepsilon}}d\lambda,\] where the last integral is finite if $\alpha>d+\beta_{BG}.$ For the case $\beta_{BG}\in(1,2],$ it holds
\[ \int\limits_{0}^{1/2}\frac{|\ln(\lambda)|^d}{\lambda^{d+2-k}}\int\limits_{2\lambda\leq|x|<1}|x|^{2-k}W(dx)\pi(d\lambda) = \int\limits_{0}^{1/2}\frac{|\ln(\lambda)|^d}{\lambda^{d+\beta_{BG}}}\int\limits_{2\lambda\leq|x|<1}\lambda^{k-2+\beta_{BG}}|x|^{2-k}W(dx)\pi(d\lambda)\]
\[\leq \int\limits_{0}^{1/2}\frac{|\ln(\lambda)|^d l(1/\lambda)}{\lambda^{d+\beta_{BG}+1-\alpha}}\int\limits_{2\lambda\leq|x|<1}|x|^{\beta_{BG}}W(dx)d\lambda \leq   C \int\limits_{0}^{1/2}\frac{d\lambda}{\lambda^{d+1+\beta_{BG}-\alpha}+2\varepsilon},\] where the last integral is finite if $\alpha>d+\beta_{BG}.$

The first integral in the case {\rm (i)} is treated analogously to the first integral in {\rm (ii)}. \end{proof}

A key tool for the analysis of supCAR fields is the following result on their marginal distributions and joint cumulant functions.

\begin{proposition}\label{prop:cum_func}
Let the supCAR random field $X(\bm{t}),$ $\bm{t}\in\mathbb{R}^d,$ with the characteristic quadruple $(0,b,W,\pi)$ satisfy the conditions of Theorem {\rm \ref{th:ex1}}. Then, $X(\cdot)$ is strongly homogeneous and isotropic, for each $\bm{t}\in\mathbb{R}^d$ it holds
\begin{equation}\label{prop:cum_func_mom}
EX(\bm{t}) = -\frac{i \pi^{\frac{d}{2}}\mathcal{K}'(0)}{\Gamma(\frac{d}{2})}\int\limits_0^\infty\frac{\pi(d\lambda)}{\lambda^{d+1}}, \ \ \ Var X(\bm{t}) = -\frac{\pi^{\frac{d}{2}}\mathcal{K}''(0)}{2^{d+2}\Gamma(\frac{d}{2})}\int\limits_0^\infty\frac{\pi(d\lambda)}{\lambda^{d+2}},
\end{equation} its marginal cumulant function has the following representation
\begin{equation}\label{eq:supcar_cum}
C(s\ddagger X(\bm{t})) = \int\limits_0^\infty\int\limits_{\mathbb{R}^d}\mathcal{K}\left( -\frac{s}{2\lambda}e^{-\lambda||\bm{u}||}  \right)d\bm{u}\pi(d\lambda), \ 
\end{equation} and its joint cumulant function is given by
\[C(s_1,...,s_k\ddagger X(\bm{t_1}),...,X(\bm{t_k})):=\log E e^{i(s_1 X(\bm{t}_1)+...+s_k X(\bm{t}_k))}\]
\begin{equation}\label{eq:supcar_cumj} 
=\int\limits_0^\infty\int\limits_{\mathbb{R}^d}\mathcal{K}\left( -\sum_{i=1}^k\frac{s_i}{2\lambda}e^{-\lambda||\bm{t}_i-\bm{u}||}  \right)d\bm{u}\pi(d\lambda),\ s_i\in\mathbb{R}, \  \bm{t}_i\in\mathbb{R}^d, \ i=\overline{1,k}, \ k\in\mathbb{N}.
\end{equation} Moreover, if there exists $p\geq1$ such that $E\Lambda(A)=0$ and $E|\Lambda(A)|^p<\infty$ for all $A\in\mathcal{S}',$ $0<|A|<\infty,$ then for all $\bm{t}\in\mathbb{R}^d$ it holds $ X(\bm{t}) \in  \overline{lin}\{\Lambda(A), A\in\mathcal{S'}\}_{L_p}.$
\end{proposition}

\begin{remark}
It follows from the expression for moments in terms of cumulants {\rm \cite[Proposition 3.3.1]{peccati2011wiener}} and representation \eqref{cum_supcar}, that for all $A\in\mathcal{S}',$ $0<|A|<\infty,$ $E\Lambda(A)=0$ if and only if $\mathcal{K}'(0)=0$ and, the finiteness of moments, $E|\Lambda(A)|^p<\infty,\ p\in\mathbb{N},$ is equivalent to $\mathcal{K}^{(m)}(0)<\infty$ for all positive integers $m\leq p.$
\end{remark}

\begin{remark}
By this result, the supCAR random field $X(\cdot)$ retains some properties of the random measure $\Lambda(\cdot).$ Indeed, if $\Lambda(\cdot)$ is $\alpha$-stable or self-decomposable, marginal distributions of $X(\cdot)$ belong to the same class of distributions.
\end{remark}

\begin{proof}
Without loss of generality, let us put $\bm{t}=0.$ Let $\{f_n(\bm{u},\lambda)\}, \bm{u}\in\mathbb{R}^d,\ \lambda\in\mathbb{R}^+,$ be a sequence of simple functions 
\[  f_n(\bm{u},\lambda) : = \sum_{i=1}^{m_n} f(\bm{u}_i^{(n)},\lambda_i^{(n)})\mathbb{1}_{A_i^{(n)}}(\bm{u},\lambda), \ (\bm{u}_i^{(n)},\lambda_i^{(n)})\in A_i^{(n)}, \ A_i^{(n)}\in\mathcal{S}',\]
 that converges to the function $f(\bm{u},\lambda)=-\frac{1}{2\lambda}e^{-\lambda||\bm{u}||}$ almost everywhere, and the following convergence holds in probability
\begin{equation}\label{eq:prob} \int\limits_0^\infty \int\limits_{\mathbb{R}^d}f_n(\bm{u},\lambda) \Lambda(d\bm{u}, d\lambda) \to -\int\limits_0^\infty \int\limits_{\mathbb{R}^d}\frac{1}{2\lambda}e^{-\lambda||\bm{u}||} \Lambda(d\bm{u}, d\lambda).\end{equation} Such a sequence of functions exists due to the definition of Rajput-Rosinski integrals and because the last integral is well-defined by Theorem~\ref{th:ex1}.

Let us consider
\[  E\exp\left\{is \int\limits_0^\infty \int\limits_{\mathbb{R}^d}f_n(\bm{u},\lambda) \Lambda(d\bm{u}, d\lambda)\right\} =  E\exp\left\{  is \sum_{i=1}^{m_n} f(\bm{u}_i^{(n)}, \lambda_i^{(n)}) \Lambda(A_i^{(n)})  \right\} \]  \[= \prod_{i=1}^{m_n} E \exp\left\{  is f(\bm{u}_i^{(n)},\lambda_i^{(n)}) \Lambda(A_i^{(n)})  \right\} = \prod_{i=1}^{m_n} \exp\left\{ C(s f(\bm{u}_i^{(n)},\lambda_i^{(n)}) \ddagger \Lambda(A_i^{(n)}))  \right\}.\] By \eqref{cum_supcar} the product equals to
\[  \prod_{i=1}^{m_n} \exp\left\{ \mathcal{K}(sf(\bm{u}_i^{(n)},\lambda_i^{(n)})) [\pi\times Leb](A_i^{(n)}) \right\} \to  \exp\left\{ \int\limits_{0}^{\infty}\int\limits_{\mathbb{R}^d}  \mathcal{K}\left(-\frac{s}{2\lambda}e^{-\lambda||\bm{u}||}\right)d\bm{u}\pi(d\lambda)\right\},\] when $n\to\infty$. Using the convergence \eqref{eq:prob}, one obtains the representation \eqref{eq:supcar_cum}.

Now, let us obtain the joint cumulant function. As
\[  s_1 X(\bm{t}_1)+...+s_k X(\bm{t}_k) = - \int\limits_0^\infty\int\limits_{\mathbb{R}^d}\sum_{i=1}^{k}\frac{s_i}{2\lambda}e^{-\lambda||\bm{t}_i-\bm{u}||}\Lambda(d\bm{u},d\lambda),\]
the formula \eqref{eq:supcar_cumj} can be obtained by following the same steps as in the marginal cumulant function case. 

To prove the strong homogeneity of the supCAR field $X(\cdot),$ let us consider its joint cumulant function and use the change of variables $\widetilde{\bm{u}} = \bm{h} - \bm{u}$
\[C(s_1,...,s_k\ddagger X(\bm{t_1}+\bm{h}),...,X(\bm{t_k}+\bm{h})) = \int\limits_0^\infty\int\limits_{\mathbb{R}^d}\mathcal{K}\left( -\sum_{i=1}^k\frac{s_i}{2\lambda}e^{-\lambda||\bm{t}_i+\bm{h}-\bm{u}||}  \right)d\bm{u}\pi(d\lambda)\]
\[=\int\limits_0^\infty\int\limits_{\mathbb{R}^d}\mathcal{K}\left( -\sum_{i=1}^k\frac{s_i}{2\lambda}e^{-\lambda||\bm{t}_i-\widetilde{\bm{u}}||}  \right)d\widetilde{\bm{u}}\pi(d\lambda)=C(s_1,...,s_k\ddagger X(\bm{t_1}),...,X(\bm{t_k})). \] Analogously, one proves the isotropy
\[C(s_1,...,s_k\ddagger X(B\bm{t_1}),...,X(B\bm{t_k})) = C(s_1,...,s_k\ddagger X(\bm{t_1}),...,X(\bm{t_k})), \] where the transformation $B$ belongs to the group of rotations $SO(\mathbb{R}^d)$ and the corresponding matrix has $det(B)=1.$

Representations \eqref{prop:cum_func_mom} follow from
\[ EX(\bm{t}) =\frac{1}{i}\frac{dC(s\ddagger X(\bm{t}))}{ds}\bigg|_{s=0}= -i\mathcal{K}^{'}(0)\int\limits_{0}^{\infty}\int\limits_{\mathbb{R}^d}\frac{1}{2\lambda}e^{-\lambda||u||}d\bm{u}\pi(d\lambda)=-\frac{i \pi^{\frac{d}{2}}\mathcal{K}'(0)}{\Gamma(\frac{d}{2})}\int\limits_0^\infty\frac{\pi(d\lambda)}{\lambda^{d+1}}, \] and
\[Var X(\bm{t})=-\frac{d^2C(s\ddagger X(\bm{t}))}{ds^2}\bigg|_{s=0} = -\mathcal{K}^{''}(0)\int\limits_{0}^{\infty}\int\limits_{\mathbb{R}^d}\frac{1}{4\lambda^2}e^{-2\lambda||u||}d\bm{u}\pi(d\lambda) = -\frac{\pi^{\frac{d}{2}}\mathcal{K}''(0)}{2^{d+2}\Gamma(\frac{d}{2})}\int\limits_0^\infty\frac{\pi(d\lambda)}{\lambda^{d+2}},\] where the last integral is calculated by applying the $d$-dimensional spherical coordinates.

Finally, by \cite[Proposition 3.6]{rajput1989spectral} it holds $X(\bm{t}) \in  \overline{lin}\{\Lambda(A), A\in\mathcal{S}'\}_{L_p(\Omega,P)},$ $\bm{t}\in\mathbb{R}^d,$ which finishes the proof of the theorem. \end{proof}

\begin{theorem}\label{Th3}
Let the supCAR random field $X(\cdot)$ satisfy the conditions of Theorem {\rm \ref{th:ex1}}, and the moments in \eqref{prop:cum_func_mom} are finite such that $EX(\bm{t})=0, \ \bm{t}\in\mathbb{R}^d.$ Then, its covariance function is
\begin{equation}\label{corsupcar} r(\bm{t}) = -c_2\mathcal{K}''(0)||\bm{t}||^{\frac{d}{2}+1}\int\limits_{0}^{\infty} \lambda^{-\frac{d}{2}-1}K_{\frac{d}{2}+1}(||\lambda \bm{t}||)\pi(d\lambda)\end{equation}
and the spectral density is
\begin{equation}\label{specsupcar}  f(\bm{\omega}) = -c_1^2\mathcal{K}''(0)\int\limits_{0}^\infty \frac{\pi(d\lambda)}{(||\bm{\omega}||^2 + \lambda^2)^{d+1}},\end{equation}where 
\begin{equation*} c_2:=\frac{c_1^2}{\Gamma(d+1)}\left(\frac{\pi}{2}\right)^{\frac{d}{2}}.
\end{equation*} 
\end{theorem}
\begin{remark}
The dependence structure of the supCAR field $X(\cdot)$ is defined by the probability measure $\pi(\cdot)$ and does not depend on the marginal distributions of the underlying CAR(1) fields $S_\lambda(\cdot)$. The covariance function and the spectral density of the supCAR field $X(\cdot)$ are weighted averages with respect to the probability measure $\pi(\cdot)$ of the covariance function {\rm \eqref{mat_cov}} and the spectral density {\rm \eqref{speccar}} of the CAR(1) field $S_\lambda(\cdot)$. 
\end{remark}

\begin{proof}
By the representation \eqref{eq:supcar_cumj} 
\[ cov(X(\mathbb{0}), X(\bm{t})) = -\frac{\partial^2}{\partial s_1 \partial s_2} C(s_1,s_2\ddagger X(\mathbb{0}), X(\bm{t}) )\bigg|_{s_1=s_2=0}\]\[= -\mathcal{K}''(0)\int\limits_{0}^{\infty}\int\limits_{\mathbb{R}^d}\frac{1}{4\lambda^2}e^{-\lambda||\bm{u}||-\lambda||\bm{t}-\bm{u}||}d\bm{u}\pi(d\lambda).\]  By denoting $g_\lambda(\bm{u}):=-e^{-\lambda||\bm{u}||}/{(2\lambda)},$ one obtains that
\begin{equation}\label{th3:conv} cov(X(\bm{0}), X(\bm{t})) = -\mathcal{K}''(0) \int\limits_{0}^\infty \{ g_\lambda * g_\lambda \}(\bm{t}) \pi(d\lambda)=-\mathcal{K}''(0) \int\limits_{0}^\infty \mathcal{F}_d^{-1}[\mathcal{F}_d^{2}[g_\lambda]](\bm{t})\pi(d\lambda),\end{equation} where $\{ g_\lambda * g_\lambda \}$ denotes a convolution, $\mathcal{F}_d[\cdot](\cdot)$ and $\mathcal{F}_d^{-1}[\cdot](\cdot)$ are the $d$-dimensional Fourier and inverse Fourier transforms respectively.

It was given in \cite{brockwell2017continuous}, that

\begin{equation}\label{th3:fourier} \mathcal{F}_d[g_\lambda](\bm{\omega}) = c_1(||\bm{\omega}||^2 + \lambda^2)^{-(d+1)/2}, \end{equation} and 
\[  \mathcal{F}_d^{-1}[\mathcal{F}^2_d[g_\lambda]](\bm{t}) = c_1^2 \left( \frac{\pi}{2} \right)^{d/2} \frac{||\lambda \bm{t}||^{d/2+1}}{|\lambda|^{d+2} \Gamma(d+1)} K_{d/2+1}(||\lambda\bm{t}||), \] which implies \eqref{corsupcar}. Applying the inverse Fourier transform, by \eqref{corsupcar} and \eqref{th3:fourier}, one obtains the spectral density \eqref{specsupcar}.
\end{proof}
\begin{example}
Let the supCAR random field $X(\cdot)$ satisfy the conditions of Theorem~{\rm \ref{th:ex1}}, and the moments in \eqref{prop:cum_func_mom} are finite such that $EX(\bm{t})=0, \ \bm{t}\in\mathbb{R}^d.$ If $\pi(\cdot)$ is the Gamma measure $\Gamma(H,1), \ H>0,$ with the density \eqref{gamma:density}, then, by Theorem~{\rm \ref{Th3}}, the covariance function of $X(\cdot)$~equals \[ r(\bm{t}) = -\frac{c_2\mathcal{K}''(0)||\bm{t}||^{\frac{d}{2}+1}}{\Gamma(H)} \int\limits_{0}^{\infty} \lambda^{H-\frac{d}{2}-2}K_{\frac{d}{2}+1}(||\bm{t}||\lambda)e^{-\lambda}d\lambda.\]
By {\rm \cite[Table 6.8]{bateman1954tables}}, one obtains
\begin{equation}\label{ex1:cor_2f1} r(\bm{t}) = -\frac{c_2\pi^{\frac{1}{2}}\mathcal{K}''(0)\Gamma(H-d-2)}{2^{H-\frac{d}{2}-1}\Gamma(H-\frac{d}{2}-\frac{1}{2})}||\bm{t}||^{d+2}{}_2 F_1\left(\frac{H+1}{2}, \frac{H}{2}; H-\frac{d}{2}-\frac{1}{2};1-||\bm{t}||^2\right),\end{equation} where ${}_2 F_1(\cdot)$ is the hypergeometric function.

As for fixed values $a,b$ and $c$, the hypergeometric function exhibits the asymptotic behaviour ${}_2F_1(a,b;c;x)\sim (-x)^{-\min(a,b)},\ x\to-\infty,$ see {\rm \cite[Proof of Theorem 2.3.2]{andrews1999special}}, the covariance function satisfies $r(\bm{t})\sim{\bm{||t||}}^{d+2-H}$ as $\bm{||t||}\to\infty,$ and $ \int_{\mathbb{R}^d}r(\bm{t})d\bm{t}<\infty$ if and only if $H>2d+2.$ Thus, the supCAR random field $X(\cdot)$ possesses long-range dependence for $H\leq 2d+2,$ and short-range dependence in the opposite case.

By \eqref{specsupcar}, the spectral density is given as
\[  f(\bm{\omega}) = -c_1^2\mathcal{K}''(0)\int\limits_0^\infty\frac{1}{(||\bm{\omega}||^2+\lambda^2)^{d+1}}\pi(d\lambda)=-\frac{c_1^2\mathcal{K}''(0)}{\Gamma(H)}\int\limits_0^{\infty}\frac{\lambda^{H-1}e^{-\lambda}}{(||\bm{\omega}||^2+\lambda^2)^{d+1}}d\lambda.\] By using $||\omega||=r$ and {\rm \cite[2.1.4.3]{prudnikov1986integrals}}, one gets
\[ f(r) = -\frac{c_1^2\mathcal{K}''(0)\Gamma(H-1)}{(-2)^{d-1}d!\Gamma(H)} \left( \frac{d}{rdr} \right)^{d}\left( r^{H-1}\left( r^{ir+iH\pi}\Gamma(2-2H,ir)\right.\right.\]\[ + \left. \left. r^{H-1}\left(-r^{ir-iH\pi}\Gamma(2-2H,-ir) \right) \right)\right),\] where $\Gamma(\cdot,\cdot)$ is the incomplete Gamma function.

Figure {\rm \ref{fig:car}} plots realisations of CAR(1) fields $S_\lambda(\cdot)$ with marginal Gamma distributions for the values $\lambda=0.5,1,1.5$ and $2.5.$ As the parameter $\lambda$ increases, realisations of the fields $S_\lambda(\cdot)$ become rougher, and variances of their marginal distributions decrease such that local spikes occur more often, but those spikes are of smaller magnitudes.
\begin{figure}[htb!]
\centering
\begin{subfigure}{7cm} %.5\textwidth
  \centering
  \includegraphics[width=\linewidth, height=0.48\linewidth]{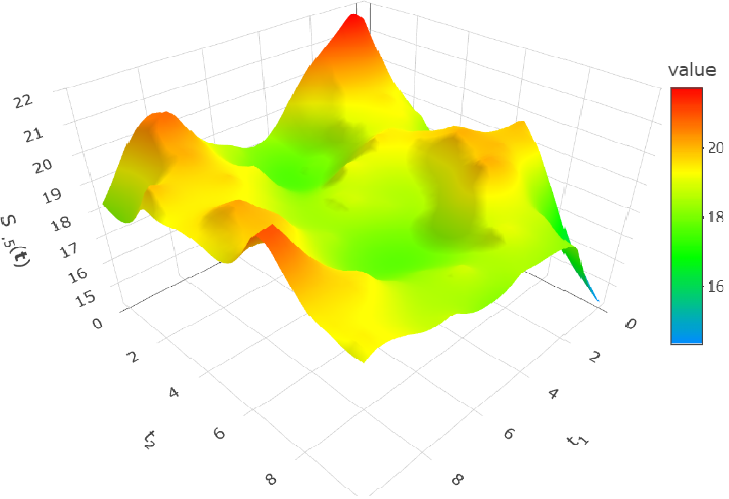} \vspace{-5mm}
  \caption{$S_{0.5}(\cdot)$}
\end{subfigure}% \vspace{-3mm}
\begin{subfigure}{7cm}
  \centering
  \includegraphics[width=\linewidth, height=0.48\linewidth]{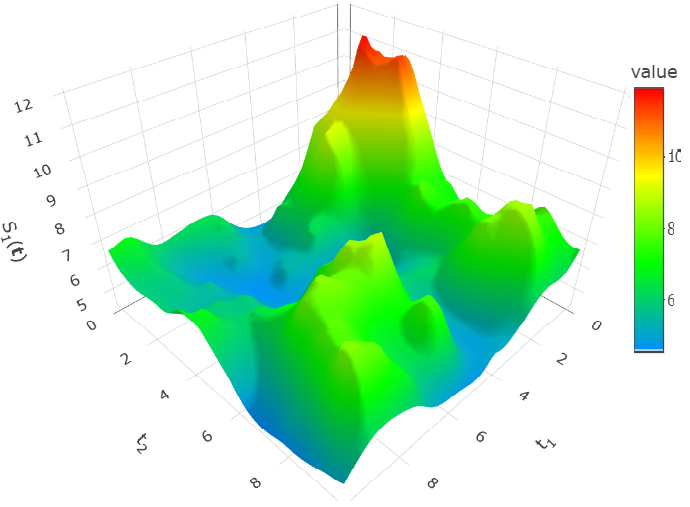} \vspace{-5mm}
  \caption{$S_{1}(\cdot)$} 
\end{subfigure}

\begin{subfigure}{7cm}
  \centering
  \includegraphics[width=\linewidth, height=0.48\linewidth]{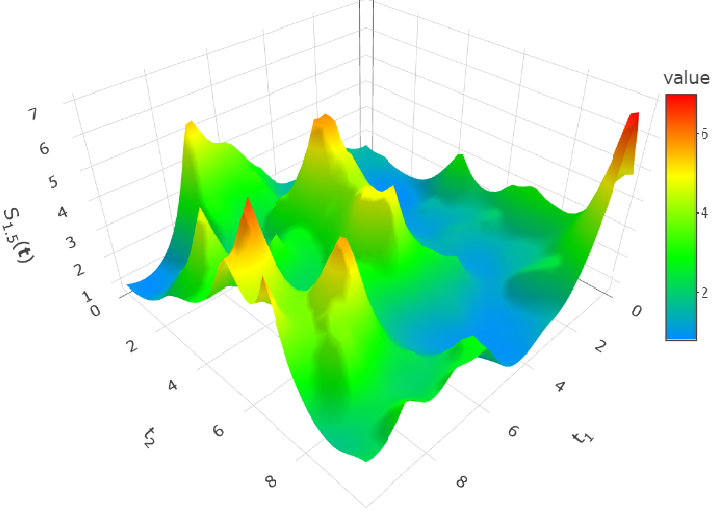}\vspace{-2mm}
  \caption{$S_{1.5}(\cdot)$}
\end{subfigure}%
\begin{subfigure}{7cm}
  \centering
  \includegraphics[width=\linewidth, height=0.48\linewidth]{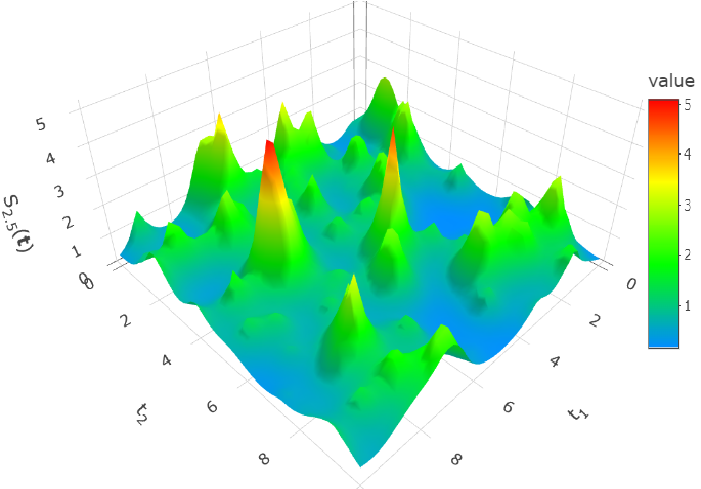}\vspace{-2mm}
  \caption{$S_{2.5}(\cdot)$}
\end{subfigure}\vspace{-2mm}
\caption{Realisations of CAR(1) fields $S_\lambda(\cdot)$ with marginal Gamma distributions}
\label{fig:car}
\end{figure}

Figure {\rm \ref{fig2:cor}} demonstrates (dotted lines) that the normalised covariance of $S_{\lambda}(\cdot)$ decreases rapidly with the increase of the parameter $\lambda$. Figures {\rm \ref{fig2:x_lrd}} and {\rm \ref{fig2:x_srd}} show the realisations of supCAR fields~$X(\cdot)$ with marginal Gamma distributions, where $\pi(\cdot)$ is given by {\rm \eqref{gamma:density}} with $H=5$ and $H=8,$ and covariances defined by \eqref{ex1:cor_2f1}. For the supCAR field~$X(\cdot)$ with $H=5,$ the underlying CAR fields $S_\lambda(\cdot)$ with smaller values of $\lambda$ receive larger weight, which leads to the long-range dependence and smoother local behaviour of  $X(\cdot),$ while for the case $H=8,$ the CAR fields $S_\lambda(\cdot)$ with larger values of $\lambda$ receive larger weights, which leads to a short-range dependence and more spiky local behaviour of  $X(\cdot).$ Figure {\rm \eqref{fig2:cor}} confirms these behaviour for the supCAR fields $X(\cdot)$ with $H=5$ and $H=8.$

\begin{figure}[htb!]
\centering
\begin{subfigure}{7cm}
  \centering
  \includegraphics[width=\linewidth, height=0.48\linewidth]{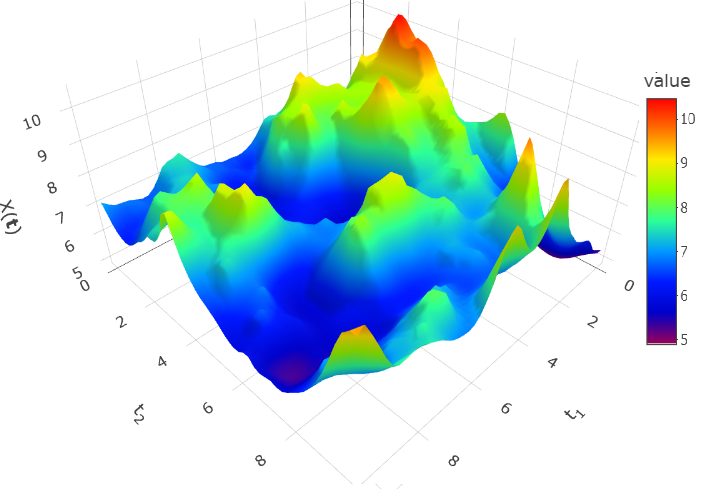}\vspace{-2mm}
  \caption{$X(\cdot)$ for $H=5$}
  \label{fig2:x_lrd}
\end{subfigure}%
\begin{subfigure}{7cm}
  \centering
  \includegraphics[width=\linewidth, height=0.46\linewidth]{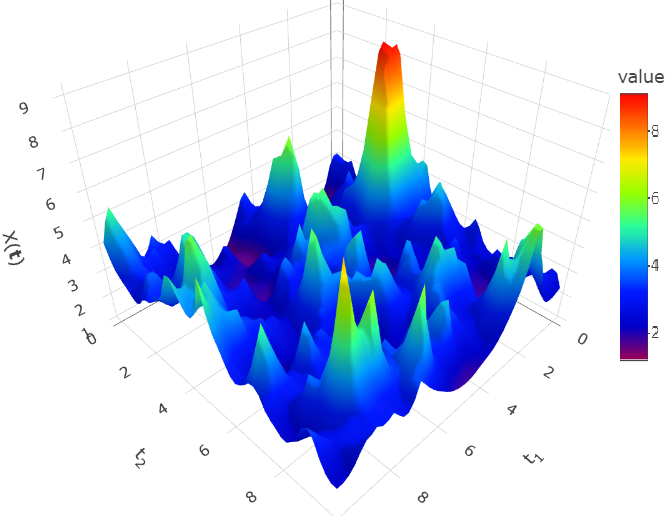}\vspace{-2mm}
  \caption{$X(\cdot)$ for $H=8$} 
  \label{fig2:x_srd}
\end{subfigure}
\caption{Realisations of $X(\cdot)$ with marginal Gamma distributions and $\pi(\cdot)$ given by~\eqref{gamma:density}}
\end{figure}

\begin{figure}[htb!]
  \centering
  \includegraphics[width=1\linewidth, height=0.3\linewidth]{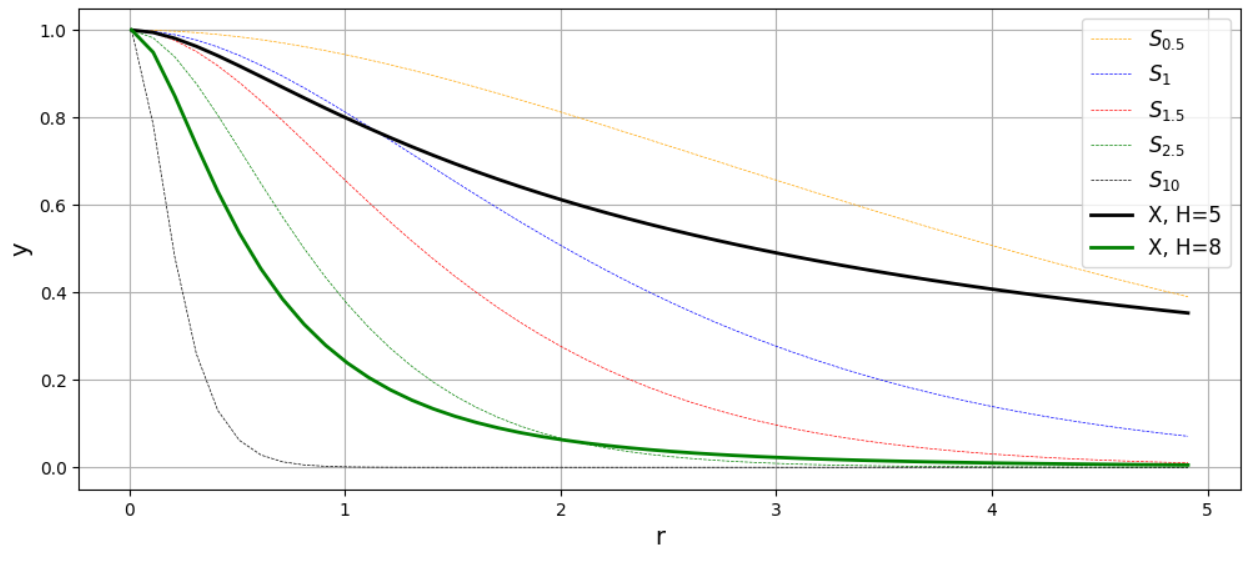}\vspace{-3mm}
  \caption{Normalised covariances of CAR (dotted lines) and supCAR (solid lines) fields}
  \label{fig2:cor}
\end{figure}
\end{example}

The following result generalises the example.
\begin{proposition}\label{prop:corasym}
Let the supCAR field $X(\cdot)$ satisfy the conditions of Theorem {\rm \ref{th:ex1}}, its first two moments are finite, $EX(\bm{t})=0, \ \bm{t}\in\mathbb{R}^d,$ and $\pi(\cdot)$ has a density $p(\lambda)=l(1/\lambda)\lambda^{\alpha-1},\ \alpha>d+2,$ where $l(\cdot)$ is a slowly varying function, strictly positive and bounded on each compact subset. Then, the covariance function of $X(\cdot)$ satisfies
\begin{equation}\label{supcarcorasym}r(\bm{t}) \sim c_3 \frac{l(||\bm{t}||)}{||\bm{t}||^{\alpha-d-2}}, \ {\mbox when} \ ||\bm{t}||\to\infty.\end{equation} If $\alpha\in(d+2,2d+2),$ the spectral density of the random field $X(\cdot)$ has a singularity at the origin and satisfies
\begin{equation}\label{supcarspecasym} f(\bm{\omega}) \sim c_4 \frac{l\left(\frac{1}{||\bm{\omega}||}\right)}{||\bm{\omega}||^{2d+2-\alpha}}, \ {\mbox when} \ \ ||\bm{\omega}||\to0,\end{equation} where 
\[ c_3 := -c_2 \mathcal{K}''(0) \int\limits_{0}^{\infty} \lambda^{\alpha-\frac{d}{2}-2}K_{\frac{d}{2}+1}(\lambda)d\lambda \ \ {\rm and} \ \ c_4 := -c_1^2 \mathcal{K}''(0)  \int\limits_0^\infty\frac{\lambda^{\alpha-1}}{(1+\lambda^2)^{d+1}}d\lambda.\]
 
\begin{proof}
Let us consider the following relation and apply \eqref{corsupcar}
\[  \frac{||\bm{t}||^{\alpha-d-2}r(\bm{t})}{l(||\bm{t}||)} =  -\frac{c_2\mathcal{K}''(0)||\bm{t}||^{\alpha-\frac{d}{2}-1}}{l(||\bm{t}||)} \int\limits_{0}^{\infty} \lambda^{-\frac{d}{2}-1}K_{\frac{d}{2}+1}(||\lambda \bm{t}||)\pi(d\lambda)\]
\begin{equation}\label{prop3:covtmp}=  -c_2\mathcal{K}''(0) \int\limits_{0}^{\infty} \lambda^{\alpha-\frac{d}{2}-2}K_{\frac{d}{2}+1}(\lambda)\frac{l\left(\frac{||\bm{t}||}{\lambda}\right)}{l(||\bm{t}||)}d\lambda.\end{equation} 
Let us show that one can use the dominated convergence theorem in the above integral. Indeed, as $\alpha>d+2,$ $K_{\frac{d}{2}+1}(\lambda)\sim \lambda^{-\frac{d}{2}-1}, \ \lambda\to0,$ $K_{\frac{d}{2}+1}(\lambda)$ decays exponentially as $\lambda\to\infty,$ and as by Potter's inequality \cite[Theorem 1.5.6]{bingham1989regular} one can choose a sufficiently small $\delta>0$ such that ${l\left(||\bm{t}||/\lambda\right)}/{l(||\bm{t}||)}\leq C(\delta) \max\{\lambda^{-\delta},\lambda^{\delta}\},$ the next integral is finite
\[ \int\limits_{0}^{\infty} \max\{
\lambda^{-\delta},\lambda^{\delta} \}\lambda^{\alpha-\frac{d}{2}-2}K_{\frac{d}{2}+1}(\lambda)d\lambda <\infty.\] Thus, by taking a limit in \eqref{prop3:covtmp}, when $||t||\to\infty,$ one gets \eqref{supcarcorasym}.
 
By \eqref{specsupcar}, it follows that
\[ \frac{||\bm{\omega}||^{2d+2-\alpha}f(\bm{\omega})}{l\left( \frac{1}{||\bm{\omega}||} \right)} = -c_1^2\mathcal{K}''(0)||\bm{\omega}||^{2d+2-\alpha}\int\limits_0^\infty\frac{\lambda^{\alpha-1}l\left(\frac{1}{\lambda}\right)}{(||\bm{\omega}||^2+\lambda^2)^{d+1} l\left( \frac{1}{||\bm{\omega}||} \right)}d\lambda \] \begin{equation}\label{prop3:sdtmp}= -c_1^2\mathcal{K}''(0) \int\limits_0^\infty\frac{\lambda^{\alpha-1}l\left(\frac{1}{\lambda||\bm{\omega}||}\right)}{(1+\lambda^2)^{d+1}l\left( \frac{1}{||\bm{\omega}||} \right)}d\lambda.\end{equation} By applying Potter's inequality and using $\alpha\in(d+2,2d+2),$ one can see that the dominated convergence theorem holds for the integral \eqref{prop3:sdtmp}, and, by taking $||\bm{\omega}||\to0,$ one gets \eqref{supcarspecasym}. \end{proof}

\end{proposition}

\section{Limit theorems for supCAR fields}\label{sect:limit}
This section derives limit theorems for integrated supCAR fields. In what follows, it is more convenient to use the alternative L\'evy-Khinchine representation of the cumulant function $\mathcal{K}(\cdot)$ in \eqref{cum_supcar1}. We assume that 
\begin{equation}\label{eq:W_x2}\int_\mathbb{R}x^2W(dx)<\infty.\end{equation} Then,
\[\mathcal{K}(s) = - \frac{b}{2}s^2 + \int_\mathbb{R}(e^{is x} - 1 - is\tau(x))W(dx) = ia_1s - \frac{b}{2}s^2 + \int_\mathbb{R}(e^{is x} - 1 - isx)W(dx),\] where $a_1=\int_\mathbb{R}(x-\tau(x))W(dx),$ see \cite[Section 8]{ken1999levy} for more details. In this case, the notation $(a_1,b,W,\pi)_1$ for the characteristic quadruple of the supCAR field $X(\cdot)$ will be used. Note also that the variance of the Gaussian component $b$ and the L\'evy measure $W(\cdot)$ remain the same. 

Without loss of generality, it will be assumed that the supCAR field $X(\cdot)$ has a zero expectation. Then, the characteristic quadruple of the supCAR field $X(\cdot)$ is of the form $(0,b,W,\pi)_1.$ Indeed, by \eqref{prop:cum_func_mom}, one gets that $\mathcal{K}'(0) = 0$ from which follows that $a_1=0,$ and the cumulant function $\mathcal{K}(\cdot)$ becomes
\begin{equation}\label{cum_supcar_one}
\mathcal{K}(s) = - \frac{b}{2}s^2 + \int\limits_\mathbb{R}(e^{is x} - 1 - isx)W(dx). 
\end{equation} 

In the following results, it is assumed that the probability measure $\pi(\cdot)$ has the density 
\begin{equation}\label{eq:meas_lrd}p(\lambda)=l(1/\lambda)\lambda^{\alpha-1},\ \alpha>d+1,\end{equation} where $l(\cdot)$ is a slowly varying at the infinity~function, strictly positive and bounded on each compact subset.

Let us introduce the integrated supCAR field 
\[ X^*(T):= \int\limits_{\Delta(T)} X(\bm{y})d\bm{y},\] where $\Delta\subset\mathbb{R}^d$ is a Jordan-measurable convex bounded set containing the origin in its interior such that $|\Delta|>0.$ Then, its boundary satisfies $|\partial\Delta|=0.$ Let $\Delta(T)$ be a homothetic transformation of $\Delta$ with the parameter $T.$ In statistical applications, $\Delta(T)$ plays a role of the expanding observation window, when $T\to\infty.$ Without loss of generality, let us assume that $\Delta\subseteq B(1).$ 

First, consider a supCAR field $X(\cdot)$ with the characteristic quadruple $(0,b,W,\pi)_1$ such that in \eqref{cum_supcar_one} $b>0$ (the Gaussian component is present) and
\begin{equation}\label{eq:meas_srd}\int_0^{\infty} \frac{\pi(d\lambda)}{\lambda^{2d+2}} < \infty,\end{equation} which implies the short-range dependence of $X(\cdot)$ by Proposition \ref{prop:corasym} as $\pi(\cdot)$ has a density $p(\lambda)=l(1/\lambda)\lambda^{\alpha-1},\ \alpha>2d+2.$ Then, the following central limit theorem~holds.

\begin{theorem}\label{th:brow}
Let $b>0$ and the measure $\pi(\cdot)$ satisfies {\rm \eqref{eq:meas_srd}}, then for $t\in[0,1]$
\begin{equation}\label{th1:conv_main} \frac{1}{\sqrt{c_5 \mathcal{K}''(0)}T^{ \frac{d}{2}}}X^*\left(t^\frac{1}{d}T\right) \overset{fdd}{{\to}}  B(t),\ {\mbox when} \ T\to\infty, \end{equation}  where $B(t), \ t\in[0,1],$ is the standard Brownian motion and \[c_5:=2|\Delta|\left(\frac{\pi^{\frac{d}{2}}\Gamma(d)}{\Gamma(\frac{d}{2})}\right)^2 \int_0^{\infty}\frac{\pi(d\lambda)}{\lambda^{2d+2}}.\]
\end{theorem}

\begin{proof}
The proof is based on the application of the method of cumulant functions \cite[Theorem 4.3]{kallenberg1997foundations}. According to the method of cumulant functions, convergence \eqref{th1:conv_main} holds if and only if the corresponding joint cumulant functions converge pointwise, when $T\to\infty$. Consider the following joint cumulant function
\begin{eqnarray}    
& &\hspace*{-1cm} C\left(s_1,...,s_k \ddagger \frac{1}{T^{\frac{d}{2}}}X^*\left(t_1^{\frac{1}{d}}T\right),...,\frac{1}{T^{\frac{d}{2}}}X^*\left(t_k^{\frac{1}{d}}T\right)\right) = \log E \exp\left(i \sum_{i=1}^k\frac{s_i}{T^{\frac{d}{2}}}X^*\left(t_i^{\frac{1}{d}}T\right)\right)  \nonumber\\
&=&\log E \exp\left(i \sum_{i=1}^k\frac{s_i}{T^{\frac{d}{2}}}\int\limits_{\Delta(t_i^{\frac{1}{d}}T)}X(\widetilde{\bm{y}})d\widetilde{\bm{y}}\right) \nonumber\\
&=& \log E \exp\left(-i \sum_{i=1}^k\frac{s_i}{T^{\frac{d}{2}}}\int\limits_{\Delta(t_i^{\frac{1}{d}}T)}\int\limits_{0}^{\infty}\int\limits_{\mathbb{R}^d}\frac{1}{2\lambda}e^{-\lambda||\widetilde{\bm{y}}-\widetilde{\bm{u}}||}\Lambda(d\widetilde{\bm{u}},d\lambda)d\widetilde{\bm{y}}\right) \nonumber\\
& = & \log E \exp\left(i \int\limits_{0}^{\infty}\int\limits_{\mathbb{R}^d}\left(-\sum_{i=1}^k\frac{s_i}{2T^{\frac{d}{2}}\lambda}\int\limits_{\Delta(t_i^{\frac{1}{d}}T)}e^{-\lambda||\widetilde{\bm{y}}-\widetilde{\bm{u}}||}d\widetilde{\bm{y}}\right)\Lambda(d\widetilde{\bm{u}},d\lambda)\right)\nonumber\\
&= &\int\limits_{0}^{\infty}\int\limits_{\mathbb{R}^d}\mathcal{K}\left(-\sum_{i=1}^k\frac{s_i}{2T^{\frac{d}{2}}\lambda}\int\limits_{\Delta(t_i^{\frac{1}{d}}T)}e^{-\lambda||\widetilde{\bm{y}}-\widetilde{\bm{u}}||}d\widetilde{\bm{y}}\right)d\widetilde{\bm{u}}\pi(d\lambda),\label{th1:cum_func} \end{eqnarray}  where $s_i\in\mathbb{R}, \  t_i\in[0,1], \ i=\overline{1,k}, \ k\in\mathbb{N},$ and the last representation follows from the steps analogous to those in the proof of Proposition \ref{prop:cum_func}. By the change of variables $T\bm{y}=\widetilde{\bm{y}}, \ T\bm{u}=\widetilde{\bm{u}},$ one obtains that the expression \eqref{th1:cum_func} equals~to
\[ \int\limits_{0}^{\infty}\int\limits_{\mathbb{R}^d}T^d\mathcal{K}\left(-\sum_{i=1}^k\frac{s_iT^{\frac{d}{2}}}{2\lambda}\int\limits_{\Delta(t_i^{\frac{1}{d}})}e^{-T\lambda||\bm{y}-\bm{u}||}d\bm{y}\right)d\bm{u}\pi(d\lambda) \]
\[ =\int\limits_{0}^{\infty}\int\limits_{B(2)}T^d\mathcal{K}\left(-\sum_{i=1}^k\frac{s_iT^{\frac{d}{2}}}{2\lambda}\int\limits_{\Delta(t_i^{\frac{1}{d}})}e^{-T\lambda||\bm{y}-\bm{u}||}d\bm{y}\right)d\bm{u}\pi(d\lambda)  \]  \begin{equation}\label{th:clt_dec}+ \int\limits_{0}^{\infty}\int\limits_{\mathbb{R}^d\setminus B(2)}T^d\mathcal{K}\left(-\sum_{i=1}^k\frac{s_iT^{\frac{d}{2}}}{2\lambda}\int\limits_{\Delta(t_i^{\frac{1}{d}})}e^{-T\lambda||\bm{y}-\bm{u}||}d\bm{y}\right)d\bm{u}\pi(d\lambda):=I_1(T) + I_2(T),\end{equation} where the splitting of the integration is motivated by geometric reasoning that will be seen later.

Let us consider the integrand in  $I_1(T).$ By $|e^{is x} - 1 - isx|\leq \frac{s^2x^2}{2},$ \eqref{cum_supcar_one} and assumption \eqref{eq:W_x2}, one obtains that 
\begin{equation}\label{eq:ks}
|\mathcal{K}(s)|\leq Cs^2.\end{equation} Thus,
\[ T^d\left|\mathcal{K}\left(-\sum_{i=1}^k\frac{s_iT^{\frac{d}{2}}}{2\lambda}\int\limits_{\Delta(t_i^{\frac{1}{d}})}e^{-T\lambda||\bm{y}-\bm{u}||}d\bm{y}\right)\right| \leq \frac{C}{\lambda^2}\left(\sum_{i=1}^ks_iT^d\int\limits_{\Delta(t_i^{\frac{1}{d}})}e^{-T\lambda||\bm{y}-\bm{u}||}d\bm{y}\right)^2\]
\begin{equation}\label{Th1:eq_asymp1}\leq \frac{C}{\lambda^{2d+2}}\left(\sum_{i=1}^ks_iT^d\lambda^d\int\limits_{\Delta(t_i^{\frac{1}{d}})}e^{-T\lambda||\bm{y}-\bm{u}||}d\bm{y}\right)^2 \leq \frac{C}{\lambda^{2d+2}},\end{equation} where the last inequality follows from the estimate
\[ T^d\lambda^d \int\limits_{\Delta(t_i^{\frac{1}{d}})}e^{-T\lambda||\bm{y}-\bm{u}||}d\bm{y} \leq \int\limits_{\mathbb{R}^d}e^{-||\bm{y}||}d\bm{y}=C<\infty. \]

As $EX(\bm{t})=0,\ \bm{t}\in\mathbb{R}^d,$  it follows from \eqref{prop:cum_func_mom} that $\mathcal{K}'(0)=0.$ By \cite[Proposition 4, p. 226]{zorich}, one gets that $\mathcal{K}(s)=-\mathcal{K}''(0)s^2+o(s^2), \ s\to0.$ Then, as for fixed $\lambda\in\mathbb{R}^+$ and $\bm{u}\in B(2)$ it holds that
\[ T^{\frac{d}{2}}\int\limits_{\Delta({t_i^{\frac{1}{d}}})} e^{-T\lambda||\bm{y}-\bm{u}||}d\bm{y}\to0,\ {\rm when} \ T\to\infty, \ i = \overline{1,k},\] we obtain that if $T\to\infty,$ then
\[T^d\mathcal{K}\left(-\sum_{i=1}^k\frac{s_iT^{\frac{d}{2}}}{2\lambda}\int\limits_{\Delta(t_i^{\frac{1}{d}})}e^{-T\lambda||\bm{y}-\bm{u}||}d\bm{y}\right)\]
\begin{equation}\label{th1:kappa_asymp}\sim -\frac{\mathcal{K}''(0)}{4\lambda^{2d+2}}\sum_{i=1}^k\sum_{j=1}^ks_i s_j{T^{2d}\lambda^{2d}}\int\limits_{\Delta(t_i^{\frac{1}{d}})}e^{-T\lambda||\bm{y}_i-\bm{u}||}d\bm{y}_i \int\limits_{\Delta(t_j^{\frac{1}{d}})}e^{-T\lambda||\bm{y}_j-\bm{u}||}d\bm{y}_j.\end{equation}

Assume that $t_i\leq t_j$ and $\bm{u}\in\Delta(t_i^\frac{1}{d}).$ Then $\bm{u}\in\Delta(t_j^\frac{1}{d}),$ and one gets that if $T\to\infty,$ then
\begin{equation}\label{intrd} {T^{2d}\lambda^{2d}}\int\limits_{\Delta(t_i^{\frac{1}{d}})}e^{-T\lambda||\bm{y}_i-\bm{u}||}d\bm{y}_i \int\limits_{\Delta(t_j^{\frac{1}{d}})}e^{-T\lambda||\bm{y}_j-\bm{u}||}d\bm{y}_j {\to} \left(\int\limits_{\mathbb{R}^d}e^{-||\bm{y}_i||}d\bm{y}_i\right)^2 = \left(\frac{2\pi^{\frac{d}{2}}\Gamma(d)}{\Gamma(\frac{d}{2})}\right)^2.\end{equation}

Now, let $t_i < t_j,$ and $\bm{u}\in\Delta(t_j^{\frac{1}{d}})\setminus \Delta(t_i^{\frac{1}{d}}).$ Denote by $\Delta_{\bm{u}}(T\lambda)$ the set $\Delta$ shifted by $\bm{u}$ and transformed homothetically with the parameter of homothety $T\lambda.$ Then,
\[ {T^{2d}\lambda^{2d}} \int\limits_{\Delta(t_i^{\frac{1}{d}})}e^{-T\lambda||\bm{y}_i-\bm{u}||}d\bm{y}_i \int\limits_{\Delta(t_j^{\frac{1}{d}})}e^{-T\lambda||\bm{y}_j-\bm{u}||}d\bm{y}_j  \] 
\[= \int\limits_{\Delta_{\bm{u}}(T\lambda t_i^{\frac{1}{d}})}e^{-||\bm{y}_i||}d\bm{y}_i \int\limits_{\Delta_{\bm{u}}(T\lambda t_j^{\frac{1}{d}})}e^{-||\bm{y}_j||}d\bm{y}_j{\to}0,\ {\rm when } \ T\to\infty,\] as $\Delta_{\bm{u}}(t_i^{\frac{1}{d}})$ does not contain the origin. Analogously, if $\bm{u}$ does not belong to $\Delta(t_j^{\frac{1}{d}}), \ t_i<t_j,$ the above integral converges to $0$. 

Thus, due to \eqref{eq:meas_srd}, \eqref{Th1:eq_asymp1} and the dominated convergence theorem it holds 
\[ \int\limits_{0}^{\infty}\frac{1}{4\lambda^{d+2}}\int\limits_{B(2)}{T^{2d}\lambda^{2d}}\int\limits_{\Delta(t_i^{\frac{1}{d}})}e^{-T\lambda||\bm{y}_i-\bm{u}||}d\bm{y}_i \int\limits_{\Delta(t_j^{\frac{1}{d}})}e^{-T\lambda||\bm{y}_j-\bm{u}||}d\bm{y}_jd\bm{u}\pi(d\lambda)\]

\[{\to}\frac{1}{4}\int\limits_0^\infty\frac{1}{\lambda^{2d+2}}\pi(d\lambda)\int\limits_{\Delta(\min\left(t_i^{\frac{1}{d}}, t_j^{\frac{1}{d}}\right))}\left(\frac{2\pi^{\frac{d}{2}}\Gamma(d)}{\Gamma(\frac{d}{2})}\right)^2d\bm{u} = \frac{c_5}{2} \min(t_i, t_j),\ {\rm when} \ T\to\infty,\] and from \eqref{th1:kappa_asymp} one obtains 
\begin{equation}\label{th1:i1lim}
I_1(T)\to-\frac{c_5\mathcal{K}''(0)}{2}\sum_{i=1}^k\sum_{j=1}^ks_is_j  \min\left( t_i, t_j\right), \ {\rm when} \  T\to\infty.
\end{equation}

Let us consider the integrand in $I_2(T).$ By the first inequality in \eqref{Th1:eq_asymp1}
\begin{equation*} T^d\left|\mathcal{K}\left(-\sum_{i=1}^k\frac{s_iT^{\frac{d}{2}}}{2\lambda}\int\limits_{\Delta(t_i^{\frac{1}{d}})}e^{-T\lambda||\bm{y}-\bm{u}||}d\bm{y}\right)\right| \leq \frac{C}{\lambda^{2d+2}}\left(\sum_{i=1}^ks_iT^d\lambda^d\int\limits_{\Delta(t_i^{\frac{1}{d}})}e^{-T\lambda||\bm{y}-\bm{u}||}d\bm{y}\right)^2.\end{equation*}  As $s_i,\ i=\overline{1,k},$ are fixed, $\bm{u}\in\mathbb{R}^d\setminus B(2)$ and $\Delta(t_i^\frac{1}{d})\subseteq B(1),$ $t_i\in[0,1],\ i=\overline{1,k},$ the right-hand side of the above inequality is bounded~by
\begin{equation}\label{th1:eq2_bound} \frac{C}{\lambda^{2d+2}}\left(\sum_{i=1}^ks_iT^d\lambda^de^{-T\lambda(||\bm{u}||-1)}\right)^2\leq\frac{C}{\lambda^{2d+2}}(T\lambda)^{2d}e^{-2T\lambda(||\bm{u}||-1)}.\end{equation} 
As $e^{-x} \leq Cx^{-d}, \ x\geq0,$ for some $C>0,$ one gets that the right-hand side of \eqref{th1:eq2_bound} can be bounded~by
\[ \frac{C}{\lambda^{2d+2}} \cdot \frac{1}{(||\bm{u}||-1)^{2d}},\] where $||\bm{u}||\geq \sqrt{2}.$ Thus, the integrand in $I_2(T)$ is bounded by the integrable function.

It holds
\[  T^d\left|\mathcal{K}\left(-\sum_{i=1}^k\frac{s_iT^{\frac{d}{2}}}{2\lambda}\int\limits_{\Delta(t_i^{\frac{1}{d}})}e^{-T\lambda||\bm{y}-\bm{u}||}d\bm{y}\right)\right| {\to} 0, \ {\rm when}\ T\to0.\] Therefore, by applying the dominated convergence theorem, one obtains that $I_2(T)\to 0,$ when $T\to\infty.$ 

Thus, from \eqref{th:clt_dec} and \eqref{th1:i1lim} and as $I_2(T)\to0,$ when $T\to\infty,$ it follows
\[ C\left(s_1,...,s_k \ddagger \frac{1}{T^{\frac{d}{2}}}X^*(t_1^{\frac{1}{d}}T),...,\frac{1}{T^{\frac{d}{2}}}X^*(t_k^{\frac{1}{d}}T)\right) {\to}  -\frac{c_5\mathcal{K}''(0)}{2}\sum_{i=1}^k\sum_{j=1}^ks_is_j  \min\left( t_i^{\frac{1}{d}}, t_j^{\frac{1}{d}}\right), \ T\to\infty,\] which is the joint cumulant function of the standard Brownian motion.
\end{proof}

Now, let us consider a supCAR field $X(\cdot)$ with the characteristic quadruple $(0,b,W,\pi)_1$ such that $b>0$ (i.e. the underlying CAR(1) fields contain Gaussian components with variances $b>0$), and measure $\pi(\cdot)$ satisfying \eqref{eq:meas_lrd} with $\alpha\in(d+2,2d+2),$ which implies the long-range dependence of $X(\cdot)$ by Proposition \ref{prop:corasym}. The next Theorem shows that under the above assumptions, the integrated supCAR process converges to the generalised Brownian motion, see~\cite[Section~2.10]{leonenko1999}. 

\begin{theorem}\label{th:herm}
Let $b>0$ and the measure $\pi(\cdot)$ satisfies \eqref{eq:meas_lrd} with $\alpha\in(d+2,2d+2).$ Then, for $t\in[0,1]$
\begin{equation}\label{eq:gen_herm} \frac{1}{\sqrt{c_6T^{3d+2-\alpha}l(T)}}X^*(t^\frac{1}{d}T) \overset{fdd}{{\to}} \int_{\mathbb{R}^d}' |\bm{y}|^{\frac{\alpha-2d-2}{2}}\int_{\Delta(t^{\frac{1}{d}})}e^{i(\bm{y},\bm{x})}d\bm{x}B(d\bm{y}), \ {when }\ T\to\infty,\end{equation} where the first integral is understood in the Wiener-It\^{o} sense {\rm \cite[Section 2.9]{leonenko1999}}, $B(\cdot)$  is a Gaussian white noise on $\mathbb{R}^d$ and $c_6 := c_1^2 b  \int\limits_0^\infty\frac{\lambda^{\alpha-1}}{(1+\lambda^2)^{d+1}}d\lambda.$
\end{theorem}

\begin{remark}
The generalised  Brownian motion in \eqref{eq:gen_herm}, which is a particular case of generalised Hermite processes~{\rm (\cite{DonhauzerOlenko2021, leonenko2014sojourn})}, appears as the limit of integrated purely Gaussian long-range dependent fields, see {\rm \cite[Theorem 2.10.2]{leonenko1999}}. Thus, Theorem {\rm\ref{th:herm}} shows that under the specified assumptions the integrated Gaussian component of the supCAR field asymptotically dominates its integrated jump component.
\end{remark}

\begin{proof}

The random measure $\Lambda(\cdot)$ can be decomposed as $\Lambda_1(\cdot) + \Lambda_2(\cdot),$ where the random measures $\Lambda_1(\cdot)$ and $\Lambda_2(\cdot)$ are independent and have characteristic quadruples $(0,b,0,\pi)_1,$ $(0,0,W,\pi)_1$ and cumulant functions \eqref{cum_supcar_one} $\mathcal{K}_1 = -\frac{b}{2}s^2,$ $\mathcal{K}_2(s) = \int_{\mathbb{R}}(e^{is x}-1-is x)W(dx)$ respectively. Thus, one can represent the supCAR field $X(\cdot)$ as
\[ X(\bm{y}) = -\int\limits_0^\infty \int\limits_{\mathbb{R}^d} \frac{1}{2\lambda}e^{-\lambda||\bm{y}-\bm{u}||}\Lambda_1(d\bm{u},d\lambda) -\int\limits_0^\infty \int\limits_{\mathbb{R}^d} \frac{1}{2\lambda}e^{-\lambda||\bm{y}-\bm{u}||}\Lambda_2(d\bm{u},d\lambda) := X_1(\bm{y})+X_2(\bm{y}),\] where $X_1(\cdot)$ and $X_2(\cdot)$ are independent random fields. 

By Proposition \ref{prop:cum_func}, the random field $X_1(\cdot)$ is strongly homogeneous, isotropic, and Gaussian. As $\mathcal{K}_1''(0) = - b,$ by Proposition \ref{prop:corasym} $X_1(\cdot)$ has a spectral density $f(\cdot)$ such that
\[f(\bm{\omega}) \sim c_6\frac{l\left(\frac{1}{||\bm{\omega}||}\right)}{||\bm{\omega}||^{2d+2-\alpha}},\ {\rm when} \ ||\bm{\omega}||\to0.\] \iffalse Let us show that $f\in L_1(\mathbb{R}^d).$ Indeed, by using the representation \eqref{specsupcar} and applying Potter's inequality, one gets that for an arbitrary small $\delta>0,$ it holds
\[f(||\bm{\omega}||)= -\frac{c_1^2\mathcal{K}''(0)l(\frac{1}{||\bm{\omega}||})}{||\bm{\omega}^{2d+2-\alpha}||} \int\limits_0^\infty\frac{\lambda^{\alpha-1}l\left(\frac{1}{\lambda||\bm{\omega}||}\right)}{(1+\lambda^2)^{d+1}l\left( \frac{1}{||\bm{\omega}||} \right)}d\lambda \] \[\leq \frac{C l\left(\frac{1}{||\bm{\omega}||}\right)}{||\bm{\omega}||^{2d+2-\alpha}}\int\limits_0^{\infty}\frac{\lambda^{\alpha-1}\max\{ \lambda^{-\delta}, \lambda^{\delta} \}}{(1+\lambda^2)^{d+1}}d\lambda\leq \frac{C l\left(\frac{1}{||\bm{\omega}||}\right)}{||\bm{\omega}||^{2d+2-\alpha}}.\] The last function belongs to the class $L_1(\mathbb{R}^d)$ due to condition \eqref{bneq}. Indeed, by putting $||\bm{\omega}||=u,$ one obtains
\[\int\limits_{\mathbb{R}^d} \frac{l\left(\frac{1}{||\bm{\omega}||}\right)}{||\bm{\omega}||^{2d+2-\alpha}} d\bm{\omega}\leq C \int\limits_0^\infty \frac{l(\frac{1}{u})}{u^{d+3-\alpha}}du= C \int\limits_0^\infty \frac{\pi(du)}{u^{d+2}}<\infty.\] \fi Thus, by the non-central limit theorem for long-range dependent Gaussian fields \cite[Theorem 5]{leonenko2014sojourn}, one gets
\begin{equation}\label{th2:gauss} \frac{1}{\sqrt{c_6T^{3d+2-\alpha}l(T)}}\int_{\Delta(t^{\frac{1}{d}}T)} X_1(\bm{y})d\bm{y} \overset{fdd}{\underset{T\to\infty}{\to}} \int_{\mathbb{R}^d}' |\bm{y}|^{\frac{\alpha-2d-2}{2}}\int_{\Delta(t^{\frac{1}{d}})}e^{i(\bm{y},\bm{x})}d\bm{x}B(d\bm{y}).\end{equation} 
Now, let us consider the following joint cumulant function
\[ C\left(s_1,...,s_k\ddagger \frac{1}{\sqrt{T^{3d+2-\alpha}l(T)}}\int\limits_{\Delta(t_1^\frac{1}{d}T)} X_2(\widetilde{\bm{y}})d\widetilde{\bm{y}},...,\frac{1}{\sqrt{T^{3d+2-\alpha}l(T)}}\int\limits_{\Delta(t_k^\frac{1}{d}T)} X_2(\widetilde{\bm{y}})d\widetilde{\bm{y}}\right)\]
\begin{equation}\label{th2:cumjump} = \int\limits_0^\infty \int\limits_{\mathbb{R}^d}\mathcal{K}_2\left( -\frac{1}{\sqrt{T^{3d+2-\alpha}l(T)}}\sum\limits_{i=1}^k\frac{s_i}{2\widetilde{\lambda}}\int\limits_{\Delta(t_i^{\frac{1}{d}}T)}e^{-\widetilde{\lambda}||\widetilde{\bm{y}}-\widetilde{\bm{u}}||}d\widetilde{\bm{y}} \right)d\widetilde{\bm{u}} \pi(d\widetilde{\lambda}):=I_2(T).\end{equation}
By the change of variables $\widetilde{\bm{y}}=\bm{y}T,\ \widetilde{\bm{u}}=\bm{u}T$ and $\widetilde{\lambda}=\lambda/T,$ the integral \eqref{th2:cumjump} equals to
\begin{equation}\label{th2:cumjump2}  \int\limits_0^\infty \int\limits_{\mathbb{R}^d}T^{d-\alpha}\mathcal{K}_2\left( -\frac{T^{\frac{\alpha-d}{2}}}{\sqrt{l(T)}}\sum\limits_{i=1}^k\frac{s_i}{2\lambda}\int\limits_{\Delta(t_i^{\frac{1}{d}})}e^{-\lambda||\bm{y}-\bm{u}||}d\bm{y} \right)d\bm{u} \lambda^{\alpha-1}l\left( \frac{T}{\lambda} \right) d\lambda. \end{equation} Let us denote
\[ \mathcal{K}_2(T, \lambda, \bm{u})  := \frac{\mathcal{K}_2\left( -\frac{T^{\frac{\alpha-d}{2}}}{\sqrt{l(T)}}\sum\limits_{i=1}^k\frac{s_i}{2\lambda}\int\limits_{\Delta(t_i^{\frac{1}{d}})}e^{-\lambda||\bm{y}-\bm{u}||}d\bm{y} \right)}{\left(  \frac{T^{\frac{\alpha-d}{2}}}{\sqrt{l(T)}}\sum\limits_{i=1}^k\frac{s_i}{2\lambda}\int\limits_{\Delta(t_i^{\frac{1}{d}})}e^{-\lambda||\bm{y}-\bm{u}||}d\bm{y} \right)^2}.\] 

Note that \eqref{th2:cumjump2} equals to
\[ \frac{1}{4} \sum\limits_{i=1}^k\sum\limits_{j=1}^k s_is_j\int\limits_0^\infty \int\limits_{\mathbb{R}^d} \mathcal{K}_2(T, \lambda, \bm{u})\int\limits_{\Delta(t_i^{\frac{1}{d}})}e^{-\lambda||\bm{y}_i-\bm{u}||}d\bm{y}_i \int\limits_{\Delta(t_j^{\frac{1}{d}})}e^{-\lambda||\bm{y}_j-\bm{u}||}d\bm{y}_j d\bm{u} \lambda^{\alpha-3}\frac{l\left( \frac{T}{\lambda} \right)}{l(T)}d\lambda\] \begin{equation}\label{th2:Iij}:= \frac{1}{4} \sum\limits_{i=1}^k\sum\limits_{j=1}^k s_is_jI_{i,j}(T).\end{equation} Now, let us show that $I_{i,j}(T)\to0, \ i,j=1,...,k,$ as $T\to\infty.$ Note that
\[  \int\limits_{\mathbb{R}^d} \int\limits_{\Delta(t_i^{\frac{1}{d}})}e^{-\lambda||\bm{y}_i-\bm{u}||}d\bm{y}_i \int\limits_{\Delta(t_j^{\frac{1}{d}})}e^{-\lambda||\bm{y}_j-\bm{u}||}d\bm{y}_j d\bm{u}\leq \int\limits_{\mathbb{R}^d} \left(\int\limits_{B(1)}e^{-\lambda||\bm{y}-\bm{u}||}d\bm{y}\right)^2d\bm{u}\]

\[ \leq  \int\limits_{B(2)} \left(\int\limits_{B(1)}e^{-\lambda||\bm{y}-\bm{u}||}d\bm{y}\right)^2d\bm{u} + \int\limits_{\mathbb{R}^d\setminus B(2)} \left(\int\limits_{B(1)}e^{-\lambda||\bm{y}-\bm{u}||}d\bm{y}\right)^2d\bm{u} \]
\[ \leq C\left(\int\limits_{B(1)}e^{-\lambda||\bm{y}||}d\bm{y}  \right)^2 +  C\int\limits_{\mathbb{R}^d\setminus B(2)} e^{-2\lambda(||\bm{u}||-1)}d\bm{u}.\] By using the $d$-dimensional spherical coordinates, the above expression can be bounded~by
\[  C\left(\int\limits_0^1s^{d-1} e^{-\lambda s} ds  \right)^2 + C\int\limits_2^\infty u^{d-1} e^{-2\lambda(u-1)} du  = \frac{C}{\lambda^{2d}} \left( \int\limits_0^\lambda s^{d-1} e^{-s} ds  \right)^2 \]
\[+\ \frac{Ce^{2\lambda}}{\lambda^d}\int\limits_{4\lambda}^{\infty} u^{d-1}e^{-u} du\ \sim \
    \begin{cases}
      \frac{1}{\lambda^d}, \ {\rm if}\ \lambda\to0,\\
      \frac{1}{\lambda^{2d}}, \ {\rm if}\ \lambda\to\infty.
    \end{cases}\
\] Note, that by \eqref{eq:ks} the function $|\mathcal{K}_2(T, \lambda, \bm{u})|$ is bounded. Thus, as $\alpha\in(d+2,2d+2),$ and by Potter's inequality for any $\delta>0$ it holds $l\left( \frac{T}{\lambda} \right)/l(T) \leq C\max\{ \lambda^{-\delta},\lambda^\delta \},$ the integrand of the external integral in $I_{i,j}(T), \ i,j=\overline{1,k},$ is bounded by the integrable function. Due to the dominated convergence theorem for fixed $\lambda\in\mathbb{R}^+$ and $\bm{u}\in\mathbb{R}^d,$ it holds $\mathcal{K}_2(T, \lambda, \bm{u})\to0,$ when $ T\to\infty.$ By applying the dominated converge theorem, one obtains that $I_{i,j}(T)\to0, $ when $ T\to\infty, \ i,j=\overline{1,k},$ from which follows that the expression in \eqref{th2:Iij} and the integral $I_2(T)$ converge to $0,$ when $T\to\infty.$ Combining this result with \eqref{th2:gauss}, finishes the proof.
\end{proof}

The cumulant function of a $\gamma$-stable random variable $Z$ with $EZ=0$ and $\gamma\in(1,2)$ allows the following representation \cite[Proof of Theorem 2.2.2]{ibragimov1971independent}

\begin{equation}\label{cum_gamma} C(s\ddagger Z) = -|s|^{\gamma} \omega(s; \gamma, \widetilde{c}_1, \widetilde{c}_2), \end{equation} where
\begin{equation*} \omega(s; \gamma, \widetilde{c}_1, \widetilde{c}_2):=
    \begin{cases}
      \frac{\Gamma(2-\gamma)}{1-\gamma}\left((\widetilde{c}_1+\widetilde{c}_2)\cos\left( \frac{\pi\gamma}{2} \right) -i(\widetilde{c}_1-\widetilde{c}_2)\sign(s)\sin\left( \frac{\pi\gamma}{2} \right)  \right), \ {\rm if} \ \gamma\neq1,\\
     \widetilde{c}_1\pi,\ {\rm if} \ \gamma=1,
    \end{cases}\
\end{equation*} with $\widetilde{c}_1,\widetilde{c}_2\geq0$ and $\widetilde{c}_1=\widetilde{c}_2$ if $\gamma=1.$

Let us consider a supCAR field $X(\cdot)$ with the characteristic quadruple $(0,0,W,\pi)_1$ ($b=0$ implies no Gaussian component in the underlying CAR(1) fields) and $\alpha\in(d+1,2d+2)$ (note that by \eqref{prop:cum_func_mom} for $\alpha\in(d+1,d+2)$ the second moments of $X(\cdot)$ are infinite). In this case, the cumulant function $\mathcal{K}(\cdot)$ equals
\begin{equation*}
\mathcal{K}(s) = \int\limits_\mathbb{R}(e^{is x} - 1 - isx)W(dx). 
\end{equation*}  Then, two possible limit behaviours are possible depending on the behaviour of the L\'evy measure $W(\cdot)$ at the~origin.

\begin{theorem}\label{th6}
Let $b=0,$ the measure $\pi(\cdot)$ satisfy \eqref{eq:meas_lrd} with $\alpha\in(d+1,2d+2),$ and, for some positive $c^+$ and $c^-,$ the L\'evy measure $W(\cdot)$ have the Blumenthal-Getoor index $\alpha/(d+1)<\beta_{BG}<\min(2, \alpha-d)$ such that
\[ W([x,\infty))\sim c^+x^{-\beta_{BG}} \ {\rm and} \ \ W((-\infty, -x])\sim c^-x^{-\beta_{BG}}, \ {\rm when} \ x\to0+.\] Then, for $t\in[0,1],$ when $T\to\infty,$ it holds
\begin{equation}\label{th3:eq_main} \frac{1}{T^{d+1-\frac{\alpha-d}{\beta_{BG}}}l(T)^\frac{1}{\beta_{BG}}}X^*(t^\frac{1}{d}T) \overset{fdd}{{\to}} \int\limits_0^\infty \int\limits_{\mathbb{R}^d}\left(-\int\limits_{\Delta(t^\frac{1}{d})} \frac{1}{2\lambda}e^{-\lambda||\bm{y}-\bm{u}||}d\bm{y}\right)K(d\bm{u},d\lambda),\end{equation} where $K(\cdot, \cdot)$ is an independently scattered scalar random measure on $\mathcal{S}'$ such that
\begin{equation*} C(s\ddagger K(d\bm{u} \times d\lambda)) =  -|s|^{\beta_{BG}} \omega(s;\beta_{BG},c^+, c^-)\lambda^{\alpha-1}d\bm{u} d\lambda.\end{equation*}
\end{theorem}
\iffalse
\begin{remark}
Note, under the conditions of Theorem {\rm \ref{th6}} $\beta_{BG}>1,$ which means that the conditions of Proposition {\rm \ref{prop:elambdast}} are satisfied and the supCAR field is well-defined. 
\end{remark}
\fi
\begin{proof} Let us consider the cumulant function
\[ C\left(s_1,...,s_k \ddagger \frac{1}{T^{d+1-\frac{\alpha-d}{\beta_{BG}}}l(T)^\frac{1}{\beta_{BG}}}X^*(t_1^{\frac{1}{d}}T),...,\frac{1}{T^{d+1-\frac{\alpha-d}{\beta_{BG}}}l(T)^\frac{1}{\beta_{BG}}}X^*(t_k^{\frac{1}{d}}T)\right) \]
\begin{equation}\label{th3:cumjump} = \int\limits_0^\infty \int\limits_{\mathbb{R}^d}\mathcal{K}\left( -\frac{1}{T^{d+1-\frac{\alpha-d}{\beta_{BG}}}l(T)^\frac{1}{\beta_{BG}}}\sum\limits_{i=1}^k\frac{s_i}{2\lambda }\int\limits_{\Delta(t_i^{\frac{1}{d}}T)}e^{-\widetilde{\lambda}||\widetilde{\bm{y}}-\widetilde{\bm{u}}||}d\widetilde{\bm{y}} \right)d\widetilde{\bm{u}} \pi(d\widetilde{\lambda}). \end{equation}
By the change of variables $\widetilde{\bm{y}} = T\bm{y}, \ \widetilde{\bm{u}} = T\bm{u}, \ \widetilde{\lambda}=\lambda/T,$ one gets that \eqref{th3:cumjump} equals to
\[ \int\limits_0^\infty \int\limits_{\mathbb{R}^d}\mathcal{K}\left( -\frac{T^\frac{\alpha-d}{\beta_{BG}}}{2\lambda l(T)^\frac{1}{\beta_{BG}}}\sum\limits_{i=1}^ks_i\int\limits_{\Delta(t_i^{\frac{1}{d}})}e^{-\lambda||\bm{y}-\bm{u}||}d\bm{y} \right)T^{d-\alpha}\lambda^{\alpha-1}l\left(\frac{T}{\lambda}\right)d\bm{u} d\lambda \]
\[  = - \int\limits_0^\infty \int\limits_{\mathbb{R}^d}\left|-\sum\limits_{i=1}^k\frac{s_i}{2\lambda}\int\limits_{\Delta(t_i^{\frac{1}{d}})}e^{-\lambda||\bm{y}-\bm{u}||}d\bm{y} \right|^{\beta_{BG}}\omega\left(-\sum\limits_{i=1}^k\frac{s_i}{2\lambda}\int\limits_{\Delta(t_i^{\frac{1}{d}})}e^{-\lambda||\bm{y}-\bm{u}||}d\bm{y};\beta_{BG}, c^+,c^-\right)\] \begin{equation}\label{th3:h}\times h_T(\lambda,\bm{u}) \lambda^{\alpha-1}d\bm{u}  d\lambda,\end{equation}  where
\begin{align} h_T(\lambda,\bm{u}):=& -\left| - \frac{T^\frac{\alpha-d}{\beta_{BG}}}{2\lambda l(T)^{\frac{1}{\beta_{BG}}}}\sum\limits_{i=1}^k s_i \int_{\Delta(t_i^{\frac{1}{d}})}e^{-\lambda||\bm{y}-\bm{u}||}d\bm{y} \right|^{-\beta_{BG}} \nonumber\\ &\times \frac{\mathcal{K}\left( -\frac{T^\frac{\alpha-d}{\beta_{BG}}}{2\lambda l(T)^\frac{1}{\beta_{BG}}}\sum\limits_{i=1}^ks_i\int_{\Delta(t_i^{\frac{1}{d}})}e^{-\lambda||\bm{y}-\bm{u}||}d\bm{y}   \right)l\left(\frac{T}{\lambda }\right)}{\omega\left( - \sum\limits_{i=1}^k\frac{s_i}{2\lambda}\int_{\Delta(t_i^{\frac{1}{d}})}e^{-\lambda||\bm{y}-\bm{u}||}d\bm{y};\beta_{BG},c^+,c^- \right)l(T)}.\nonumber
\end{align}
By \cite[Eq. (41)]{philippe2014contemporaneous}, it holds
\begin{equation}\label{th3:kappa} \lim\limits_{a\to0} a \mathcal{K}(a^{-\frac{1}{\beta_{BG}}}s ) = -|s|^{\beta_{BG}}\omega(s; \beta, c^+, c^-),\ s\in\mathbb{R}.\end{equation} As $\alpha>d+1,$ by \eqref{th3:kappa}, for fixed $\lambda\in\mathbb{R}^+,\ \bm{u}\in\mathbb{R}^d,$ it holds $h_T(\lambda,\bm{u})\to1,$ when $T\to\infty.$

Let us show that one can use the dominated convergence theorem in \eqref{th3:h}. Note that $ \omega(s;\beta_{BG},c^+,c^-)$ depends only on $\sign(s),$ so its bounded. By \cite[Eq. (42)]{philippe2014contemporaneous}, it holds $|\mathcal{K}(s)|\leq C|s|^{\beta_{BG}},$ and as by Potter's inequality ${l\left(\frac{T}{\lambda}\right)}/{l(T)}\leq C\max\{\lambda^{-\delta}, \lambda^\delta\},$ one gets that $h_T(\lambda,\bm{u})\leq C\max\{\lambda^{-\delta}, \lambda^\delta\}$ for any $\delta>0.$ 

On the other hand, as $s_i\in\mathbb{R}, \ i=1,...,k,$ are fixed, $\Delta(t_i)\subset B(1), \ i=1,...,k,$ and $\beta_{BG}>1,$ one~obtains
\[ \int\limits_{\mathbb{R}^d}\left|-\sum\limits_{i=1}^k \frac{s_i}{2\lambda} \int\limits_{\Delta(t_i^{\frac{1}{d}})}e^{-\lambda||\bm{y}-\bm{u}||}d\bm{y} \right|^{\beta_{BG}}d\bm{u} \leq \frac{C}{\lambda^{\beta_{BG}}}\int\limits_{\mathbb{R}^d}\left( \int\limits_{B(1)}e^{-\lambda||\bm{y}-\bm{u}||}d\bm{y}   \right)^{\beta_{BG}}d\bm{u}\] \[= \frac{C}{\lambda^{\beta_{BG}}}\left(  \int\limits_{B(2)} \left(\int\limits_{B(1)}e^{-\lambda||\bm{y}-\bm{u}||}d\bm{y}\right)^{\beta_{BG}} d\bm{u} + \int\limits_{\mathbb{R}^d\setminus B(2)} \left(\int\limits_{B(1)}e^{-\lambda||\bm{y}-\bm{u}||}d\bm{y}\right)^{\beta_{BG}} d\bm{u} \right) \] \[\leq \frac{C}{\lambda^{\beta_{BG}}} \left(  \left(  \int\limits_0^1s^{d-1}e^{-\lambda s}ds \right)^{\beta_{BG}} +   \int\limits_2^\infty u^{d-1}e^{-\lambda(u-1)\beta_{BG}} du \right) \]
\[ = \frac{C}{\lambda^{\beta_{BG}}} \left( \frac{1}{\lambda^{d \beta_{BG}}}\left(\int\limits_0^\lambda s^{d-1}e^{-s}ds \right)^{\beta_{BG}} + \frac{e^{\lambda \beta_{BG}}}{(\beta_{BG}\lambda)^d}\int\limits_{2\beta_{BG}\lambda}^\infty u^{d-1}e^{-u}du \right) \]

\begin{equation*}\sim
    \begin{cases}
      \frac{1}{\lambda^{\beta_{BG}+d}}, \ {\rm if} \ \lambda\to0,\\
      \frac{1}{\lambda^{{\beta_{BG}(1+d)}}}, \ {\rm if} \ \lambda\to\infty.
    \end{cases}\
\end{equation*}  Thus, if $\alpha/(d+1)<\beta_{BG}<\min(2,\alpha-d)$ the dominated convergence theorem holds for the integral \eqref{th3:h} and it converges to
\[ -\int\limits_0^\infty \int\limits_{\mathbb{R}^d}\left| -\sum\limits_{i=1}^k \frac{s_i}{2\lambda} \int\limits_{\Delta(t_i^{\frac{1}{d}})}e^{-\lambda||\bm{y}-\bm{u}||}d\bm{y} \right|^{\beta_{BG}}\omega\left(  -\sum\limits_{i=1}^k \frac{s_i}{2\lambda} \int\limits_{\Delta(t_i^{\frac{1}{d}})}e^{-\lambda||\bm{y}-\bm{u}||}d\bm{y}; \beta_{BG}, c^+, c^-  \right)\lambda^{\alpha-1}d\bm{u}d\lambda,\] when $T\to\infty.$ 

Analogously to the proof of Proposition \ref{prop:cum_func}, one shows that the above expression is the joint cumulant function of the limit process in \eqref{th3:eq_main}, which finishes the proof.\end{proof}

The de Bruin conjugate \cite[Theorem 1.5.13]{bingham1989regular} of a slowly varying function $l(\cdot)$ is a unique slowly varying function $l^{\#}(\cdot)$ satisfying the conditions
\begin{equation}\label{debruinconj} l(x)l^{\#}(xl(x))\to1 \ {\rm and} \ \ l^{\#}(x)l(xl^{\#}(x))\to1,  \ {\rm when} \ x\to\infty. \end{equation}

\begin{theorem}\label{th7}
Let $b=0,$ the measure $\pi(\cdot)$ satisfy \eqref{eq:meas_lrd} with $\alpha\in(d+1,2d+2),$ and, for some positive $c^+$ and $c^-,$ the L\'evy measure $W(\cdot)$ have the Blumenthal-Getoor index $0<\beta_{BG}<{\alpha}/{(d+1)}$ such that
\[ W([x,\infty))\sim c^+x^{-\beta_{BG}}, \ {\rm and} \ \ W((-\infty, -x])\sim c^-x^{-\beta_{BG}},  \ {\rm when} \ x\to0+.\] Then, for $t\in[0,1]$ it holds true
\begin{equation*} \frac{1}{c_7^{\frac{\alpha}{d+1}}T^{\frac{d(d+1)}{\alpha}}l^{\#}(T)^\frac{d(d+1)}{\alpha}}X^*(t^\frac{1}{d}T) \overset{fdd}{{\to}} L_{\frac{\alpha}{d+1}}(t), \ { when} \ T\to\infty, \end{equation*} where $l^{\#}(\cdot)$ is the de Bruin conjugate of $l^{-\frac{1}{d}}(x^\frac{d}{\alpha}),$ and $L_{\frac{\alpha}{d+1}}(\cdot)$ is ${\frac{\alpha}{d+1}}$-stable L\'evy process such~that
\begin{equation}\label{th4:levy_cum} C\left(s\ddagger L_{\frac{\alpha}{d+1}}(t)\right) = -t|s|^{\frac{\alpha}{d+1}}\omega\left(s; \frac{\alpha}{d+1}, c^{-}_\frac{\alpha}{d+1}, c^{+}_\frac{\alpha}{d+1}\right),\end{equation}
\[ c^{-}_\frac{\alpha}{d+1} := \int\limits_{-\infty}^0|y|^{\frac{\alpha}{d+1}}W(dy), \quad c^{+}_\frac{\alpha}{d+1} := \int\limits_{0}^\infty|y|^{\frac{\alpha}{d+1}}W(dy),\quad c_7 := \frac{\pi^{\frac{d}{2}}\Gamma(d)|\Delta|^{\frac{d+1}{\alpha}}}{\alpha^{\frac{d+1}{\alpha}}\Gamma(\frac{d}{2})}.\]
\end{theorem}

\begin{proof}
It is enough to prove that for all $a_i\in\mathbb{R}, \ i=\overline{1,k-1},$ and  $0\leq t_1\leq ...\leq t_k\leq1,$ $k\in\mathbb{N},$ it holds
\begin{equation}\label{th4:conv} \sum\limits_{i=1}^{k-1} a_iA_T^{-1}\left(X^*\left(t_{i+1}^{\frac{1}{d}}T\right)-X^*\left(t_{i}^{\frac{1}{d}}T\right)\right)  \overset{d}{{\to}} \ c_7 \sum\limits_{i=1}^{k-1} a_i(L(t_{i+1})-L(t_{i})), \ {\rm when} \ T\to\infty,\end{equation} where 
$A_T:=T^{\frac{d(d+1)}{\alpha}}l^{\#}(T)^\frac{d(d+1)}{\alpha}.$ 

Let us consider the following joint cumulant function
\[ C\left(s_1,...,s_{k-1}\ddagger a_1A_T^{-1}\left(X^*\left(t_{2}^{\frac{1}{d}}T\right)-X^*\left(t_{1}^{\frac{1}{d}}T\right)\right),...,a_{k-1}A_T^{-1}\left(X^*\left(t_{k}^{\frac{1}{d}}T\right)-X^*\left(t_{k-1}^{\frac{1}{d}}T\right)\right)\right)  \]
\[= \int\limits_0^\infty \int\limits_{\mathbb{R}^d}\mathcal{K}\left( -\frac{1}{A_T}\sum\limits_{i=1}^{k-1}\frac{a_is_i}{2\widetilde{\lambda} }\int\limits_{\Delta(t_{i+1}^{\frac{1}{d}}T) \setminus \Delta(t_{i}^{\frac{1}{d}}T)}e^{-\widetilde{\lambda}||\widetilde{\bm{y}}-\widetilde{\bm{u}}||}d\bm{y} \right)d\widetilde{\bm{u}}\,  \widetilde{\lambda}^{\alpha-1}l\left( \frac{1}{\widetilde{\lambda}} \right)d\widetilde{\lambda}. \] By the change of variables $\widetilde{\bm{y}} = T\bm{y}, \ \widetilde{\bm{u}} = T\bm{u},$ and $\widetilde{\lambda} = \lambda'/A_T^{\frac{1}{d+1}},$ one obtains that the above cumulant function equals to  
\[ \int\limits_0^\infty \int\limits_{\mathbb{R}^d}\mathcal{K}\left( -\sum\limits_{i=1}^{k-1}\frac{T^da_is_i}{2A_T^\frac{d}{d+1}\lambda'}\int\limits_{\Delta(t_{i+1}^{\frac{1}{d}}) \setminus \Delta(t_{i}^{\frac{1}{d}})}e^{-\frac{T\lambda'||\bm{y}-\bm{u}||}{A_T^\frac{1}{d+1}}}d\bm{y} \right)d\bm{u}\, (\lambda')^{\alpha-1}\frac{l\Big( \frac{A_T^\frac{1}{d+1}}{\lambda'}\Big)}{l^{\#}(T)^d}d\lambda'. \]
By the change of variables $\lambda' = \lambda^{-\frac{1}{d+1}},$ it takes the following form 
\[\frac{1}{d+1} \int\limits_0^{\infty} \int\limits_{\mathbb{R}^d} \mathcal{K}\left(- \lambda\sum\limits_{i=1}^{k-1} \frac{T^da_is_i}{2(A_T\lambda)^{\frac{d}{d+1}}}\ \int\limits_{\Delta(t_{i+1}^{\frac{1}{d}}) \setminus \Delta(t_{i}^{\frac{1}{d}})} e^{-\frac{T||\bm{y}-\bm{u}||}{(A_T\lambda)^{\frac{1}{d+1}}}}d\bm{y} \right) d\bm{u} \,\frac{l\left((A_T\lambda)^{\frac{1}{d+1}} \right)}{l^{\#}(T)^d} \lambda^{-\frac{\alpha}{d+1}-1}d\lambda\] \begin{equation}\label{th4:cumjump}:= I_1(T) + I_2(T)+I_3(T),\end{equation} where the integration sets in $I_1(T),$ $I_2(T)$ and $I_3(T)$ are $\mathbb{R}^+\times \widetilde{\Delta},$ $\mathbb{R}^+\times(B(2)\setminus \widetilde{\Delta}),$ and $\mathbb{R}^+\times (\mathbb{R}^d\setminus B(2))$ respectively, and  $\widetilde{\Delta} = \cup_{i=1}^{k-1}(\Delta(t_{i+1}^{\frac{1}{d}})\setminus \Delta(t_{i}^{\frac{1}{d}}))=\Delta(t_{k}^{\frac{1}{d}})\setminus \Delta(t_{1}^{\frac{1}{d}}).$

Let a fixed integer $i$ belongs to $\{1,2,...,k-1\}.$ Consider the following integral
\begin{equation}\label{th4:int_tmp} \frac{T^d}{2(A_T\lambda)^{\frac{d}{d+1}}}\int\limits_{\Delta(t_{i+1}^{\frac{1}{d}}) \setminus \Delta(t_{i}^{\frac{1}{d}})} e^{-\frac{T||\bm{y}-\bm{u}||}{(A_T\lambda)^{\frac{1}{d+1}}}}d\bm{y} .\end{equation} 

If $\bm{u}\not\in \Delta(t_{i+1}^\frac{1}{d}) \setminus \Delta(t_{i}^{\frac{1}{d}}),$ and $c$ denotes the distance between $\bm{u}$ and the set $\Delta(t_{i+1}^\frac{1}{d}) \setminus \Delta(t_{i}^{\frac{1}{d}}),$ then, for each fixed $\lambda\in \mathbb{R}^+$ the expression \eqref{th4:int_tmp} is bounded by 
\[  \frac{ C T^d}{2(A_T\lambda)^{\frac{d}{d+1}}} e^{-\frac{Tc}{(A_T\lambda)^{\frac{1}{d+1}}}} \to 0, \ {\rm when}\ T\to\infty,\] as ${T}{A_T^{-\frac{1}{d+1}}}\to\infty,$ when $T\to\infty.$ 

If $\bm{u}$ belongs to the interior of $\Delta(t_{i+1}^\frac{1}{d}) \setminus \Delta(t_{i}^{\frac{1}{d}}),$ then there exists a ball $B_{\bm{u}}(C)\subset  \Delta(t_{i+1}^\frac{1}{d}) \setminus \Delta(t_{i}^{\frac{1}{d}})$ and \eqref{th4:int_tmp} equals to
\begin{equation}\label{th7:conv_const} \frac{ T^d}{2(A_T\lambda)^{\frac{d}{d+1}}} \left(\int\limits_{B_{\bm{u}}(C)} e^{-\frac{T||\bm{y}-\bm{u}||}{(A_T\lambda)^{\frac{1}{d+1}}}}d\bm{y} + \int\limits_{(\Delta(t_{i+1}^\frac{1}{d}) \setminus \Delta(t_{i}^{\frac{1}{d}}))\setminus B_{\bm{u}}(C)} e^{-\frac{T||\bm{y}-\bm{u}||}{(A_T\lambda)^{\frac{1}{d+1}}}}d\bm{y} \right)\to \frac{\pi^{\frac{d}{2}}\Gamma(d)}{\Gamma(\frac{d}{2})}, \end{equation}
when $T\to\infty.$ The proof is analogous to \eqref{intrd}, as the first integral converges to ${\pi^{\frac{d}{2}}\Gamma(d)}/{\Gamma(\frac{d}{2})}$ and the second integral converges to $0,$ because the integration set does not contain the origin. 

For $\bm{u}\in \partial(\Delta(t_{i+1}^{\frac{1}{d}})\setminus \Delta(t_{i}^{\frac{1}{d}})),$  the expression \eqref{th4:int_tmp} is finite, but its value is not important for the subsequent computations as the corresponding $\bm{u}$-integration set is of measure zero, i.e. $|\partial(\Delta(t_{i+1}^{\frac{1}{d}})\setminus \Delta(t_{i}^{\frac{1}{d}}))|~=~0.$

As $l^{\#}(\cdot)$ is the de Bruin conjugate of $l^{-\frac{1}{d}}(x^\frac{d}{\alpha}),$ by \eqref{debruinconj},  for each fixed $\lambda\in\mathbb{R}^+$ it holds
\begin{equation}\label{th4:eq_bruin_conv} \frac{l\left( (A_T\lambda)^{\frac{1}{d+1}} \right)}{l^{\#}(T)^d} = \frac{l\left( (A_T\lambda)^{\frac{1}{d+1}} \right)}{l(A_T^{\frac{1}{d+1}})}\frac{l(A_T^{\frac{1}{d+1}})}{l^{\#}(T)^d}\to1, \ {\rm when}\ T\to\infty. \end{equation} Thus, changing the order of the limit and the integration in $I_1(T),$ by \eqref{th7:conv_const} and \eqref{th4:eq_bruin_conv}, it follows that, when $T\to\infty,$ the integral $I_1(T)$ converges to
\begin{equation}\label{th4:lim} \frac{1}{d+1} \int\limits_0^{\infty} \int\limits_{\widetilde{\Delta}} \mathcal{K}\left(-\frac{\pi^{\frac{d}{2}}\Gamma(d)\lambda}{\Gamma(\frac{d}{2})} \sum\limits_{i=1}^{k-1} a_i s_i \mathbb{1}_{\Delta(t_{i+1}^\frac{1}{d}) \setminus \Delta(t_{i}^{\frac{1}{d}})} (\bm{u}) \right) d\bm{u} \lambda^{-\frac{\alpha}{d+1}-1}d\lambda.\end{equation} As  $0\leq t_1\leq...\leq t_k\leq1,$ the above equals to \[ \frac{1}{d+1} \sum\limits_{i=1}^{k-1}  \int\limits_0^{\infty} \int\limits_{\Delta(t_{i+1}^\frac{1}{d}) \setminus \Delta(t_{i}^{\frac{1}{d}})} \mathcal{K}\left(-\frac{\pi^{\frac{d}{2}}\Gamma(d)a_is_i\lambda}{\Gamma(\frac{d}{2})}  \right) d\bm{u} \lambda^{-\frac{\alpha}{d+1}-1}d\lambda \]
\[ = \frac{|\Delta|}{d+1} \sum\limits_{i=1}^{k-1} (t_{i+1}-t_{i})\int\limits_0^\infty \mathcal{K}\left(-\frac{\pi^{\frac{d}{2}}\Gamma(d)a_is_i\lambda}{\Gamma(\frac{d}{2})}  \right) \lambda^{-\frac{\alpha}{d+1}-1}d\lambda \]
\begin{equation}\label{th4:eqint} = c_7^{\frac{\alpha}{d+1}} \sum\limits_{i=1}^{k-1} (t_{i+1}-t_{i})|a_is_i|^\frac{\alpha}{d+1}\frac{\alpha}{d+1} \int\limits_0^\infty \mathcal{K}\left(-\sign(a_is_i)\lambda \right) \lambda^{-\frac{\alpha}{d+1}-1}d\lambda .\end{equation} Let us consider the integral in the above expression, when $\sign(a_is_i)=1$ (the case $\sign(a_is_i)=-1$ is treated analogously, and the case $\sign(a_is_i)=0$ is trivial as $\mathcal{K}(0)=0$). 

As by \cite[Theorem 2.2.2]{ibragimov1971independent}
\[ \int\limits_0^\infty (e^{\mp i u}-1\pm iu)u^{-\gamma-1} du=\exp\left\{ \mp \frac{i\pi\gamma}{2} \right\}\frac{\Gamma(2-\gamma)}{\gamma(\gamma-1)}, \ 1<\gamma<2,\] one obtains
\[\frac{\alpha}{d+1} \int\limits_0^\infty \mathcal{K}\left(-\lambda \right) \lambda^{-\frac{\alpha}{d+1}-1}d\lambda =  \frac{\alpha}{d+1} \int\limits_{-\infty}^\infty \int\limits_0^\infty (e^{-i\lambda y}-1+i\lambda y)\lambda^{-\frac{\alpha}{d+1}-1}d\lambda W(dy)\]
\begin{eqnarray} & = &\frac{\alpha}{d+1} \int\limits_{0}^\infty \int\limits_0^\infty (e^{-i\lambda }-1+i\lambda )\lambda^{-\frac{\alpha}{d+1}-1}d\lambda \, y^\frac{\alpha}{d+1} W(dy) \nonumber\\ &+& \frac{\alpha}{d+1} \int\limits_{-\infty}^0 \int\limits_0^\infty (e^{i\lambda }-1-i\lambda )\lambda^{-\frac{\alpha}{d+1}-1}d\lambda\, |y|^\frac{\alpha}{d+1} W(dy)\nonumber\\ &= &\frac{\Gamma(2-\frac{\alpha}{d+1})}{\left(\frac{\alpha}{d+1}-1\right)}\left(\exp\left( - \frac{i\pi\alpha}{2(d+1)} \right)\int\limits_0^\infty y^{\frac{\alpha}{d+1}}W(dy)  +  \exp\left(\frac{i \pi \alpha}{2(d+1)}\right)\int\limits_{-\infty}^0 |y|^{\frac{\alpha}{d+1}}W(dy)\right)\nonumber\\ & =& -\frac{\Gamma(2-\frac{\alpha}{d+1})}{\left(1-\frac{\alpha}{d+1}\right)}\left(  \cos\left( \frac{ \pi \alpha}{2(d+1)} \right)\left(\int\limits_0^\infty y^{\frac{\alpha}{d+1}} W(dy) + \int\limits_{-\infty}^0|y|^{\frac{\alpha}{d+1}} W(dy)\right) \right.\nonumber\\ & -& \left.i\sin\left( \frac{ \pi \alpha}{2(d+1)} \right)\left(\int\limits_0^\infty y^{\frac{\alpha}{d+1}} W(dy) - \int\limits_{-\infty}^0|y|^{\frac{\alpha}{d+1}} W(dy)\right)\right)  =-\omega\left(a_is_i; \frac{\alpha}{d+1}, c^{-}_\frac{\alpha}{d+1}, c^{+}_\frac{\alpha}{d+1}\right).\nonumber \end{eqnarray} Thus, by \eqref{cum_gamma} and \eqref{th4:eqint} and the above result, it follows
\[ I_1(T)\to- c_7^{\frac{\alpha}{d+1}} \sum\limits_{i=1}^{k-1} (t_{i+1}-t_{i})|a_is_i|^\frac{\alpha}{d+1}\omega\left(a_is_i; \frac{\alpha}{d+1}, c^{-}_\frac{\alpha}{d+1}, c^{+}_\frac{\alpha}{d+1}\right), \ {\rm \ when}\ T\to\infty,\] which by the independence of the increments of the process $L_{\frac{\alpha}{d+1}}(\cdot)$ and \eqref{th4:levy_cum} is the joint cumulant function of the limit in \eqref{th4:conv}. 

Let us show that the dominated convergence theorem holds true for the integral $I_1(T)$. By Potter's inequality and \eqref{th4:eq_bruin_conv}, for all $\lambda\in\mathbb{R}^+$ and any $\delta>0$ it holds
\begin{equation}\label{th4:eq_bruin_potter} \frac{l\left((A_T\lambda)^{\frac{1}{d+1}} \right)}{l^{\#}(T)^d} = \frac{l\left( (A_T\lambda)^{\frac{1}{d+1}} \right)}{l(A_T^{\frac{1}{d+1}})}\frac{l(A_T^{\frac{1}{d+1}})}{l^{\#}(T)^d}\leq C\max\{ \lambda^{-\delta}, \lambda^{\delta} \},\end{equation} when $T$ is sufficiently large. 

It also holds
\[
|\mathcal{K}(\lambda)| \leq \int\limits_{|\lambda y|\leq1}|e^{i \lambda y} - 1 - i\lambda y|W(dy) + \int\limits_{|\lambda y|>1}|e^{i \lambda y} - 1 - i \lambda y|W(dy)
\]\[\leq \int\limits_{|\lambda y|\leq1}|\lambda y|^2 W(dy) + \int\limits_{|\lambda y| > 1}(2 + |\lambda y|)W(dy) \leq \lambda ^2 \int\limits_{|y|\leq\frac{1}{|\lambda|}} |y|^2 W(dy) +  3|\lambda|\int\limits_{|y|>\frac{1}{|\lambda|}} |y| W(dy)\] 
\begin{equation}\label{th4:kappa_est}
=:\mathcal{K}_1(\lambda)+\mathcal{K}_2(\lambda).
\end{equation}

Due to \eqref{eq:W_x2} and properties of the Blumental-Getoor index, chosing $\gamma'\in(0,1)$ such that $\beta_{BG}<1+\gamma'$ and a sufficiently small $\delta>0,$ one obtains
\[ \int\limits_0^\infty\mathcal{K}_1(-\lambda)\lambda^{-\gamma'-2}\max\left( \lambda^{-\delta}, \lambda^{\delta} \right) d\lambda \] \[\leq \int\limits_{|y|\leq1}y^2W(dy)\int\limits_0^1\lambda^{-\gamma'-\delta}d\lambda + \int\limits_{|y|\leq1}y^2 \int\limits_{1}^{\frac{1}{|y|}}\lambda^{-\gamma'+\delta}d\lambda W(dy) + \int\limits_{|y|>1}y^2 \int\limits_0^{\frac{1}{|y|}} \lambda^{-\gamma'-\delta}d\lambda W(dy)\]
\begin{equation} \label{th4:eqkappa1} = C_1 +C_2 \int\limits_{|y|\leq1}|y|^{1+\gamma'-\delta}W(dy) + C_3\int\limits_{|y|>1} |y|^{1+\gamma'+\delta}W(dy) < \infty.\end{equation} Analogously, 
\begin{eqnarray}  & &\hspace*{-1cm}\int\limits_0^\infty \mathcal{K}_2(-\lambda)\lambda^{-\gamma'-2} \max\left(\lambda^{-\delta}, \lambda^{\delta} \right) d\lambda\nonumber\\  &=& 3\Bigg(\int\limits_{|y|\leq1}|y|\int\limits_{\frac{1}{|y|}}^\infty \lambda^{-\gamma'-1+\delta}d\lambda W(dy) + \int\limits_{|y|>1}|y|W(dy)\int\limits_{1}^\infty \lambda^{-\gamma'-1+\delta} d\lambda \nonumber\\ 
&+& \int\limits_{|y|>1}|y|\int\limits_{\frac{1}{|y|}}^1\lambda^{-\gamma'-1-\delta}d\lambda W(dy)\Bigg)
   \leq C_4 + C_5 \int\limits_{|y|>1}|y|^{1+\gamma'+\delta} W(dy) < \infty. \label{th4:eqkappa2} \end{eqnarray}

Let us denote
\begin{equation}\label{th7:gt} g_T(\lambda,\bm{u}):=-\sum\limits_{i=1}^{k-1} \frac{T^da_is_i}{2(A_T\lambda)^{\frac{d}{d+1}}}\  \int\limits_{\Delta(t_{i+1}^{\frac{1}{d}}) \setminus \Delta(t_{i}^{\frac{1}{d}})} e^{-\frac{T||\bm{y}-\bm{u}||}{(A_T\lambda)^{\frac{1}{d+1}}}}d\bm{y}.\end{equation} 
Using \eqref{th4:eq_bruin_potter}, the integrand in $I_1(T)$ can be bounded by
\begin{equation}\label{th4:eq_kappa1}
\begin{split}
|\mathcal{K}\left(- \lambda g_T(\lambda,\bm{u}) \right)| \max\left( \lambda^{-\delta}, \lambda^\delta \right) \lambda^{-\frac{\alpha}{d+1}-1}\mathbb{1}\left( \lambda\leq\frac{T^{d+1}}{A_T}\right) \\ 
+ |\mathcal{K}\left(- \lambda g_T(\lambda,\bm{u}) \right)| \max\left( \lambda^{-\delta}, \lambda^\delta \right) \lambda^{-\frac{\alpha}{d+1}-1}\mathbb{1}\left(\lambda>\frac{T^{d+1}}{A_T}\right). 
\end{split}
\end{equation} 
Let us show that for fixed $a_i,s_i\in\mathbb{R},$ $t_i\in[0,1],\ i=1,...,k-1,$ and all $\lambda\leq {T^{d+1}}/{A_T}$ and $\bm{u}\in B(2),$ there exist constants $C_1$ and $C_2$ such that $0<C_1\leq C_2<\infty$ and
\begin{equation}\label{th7:gt_const} -C_2 \leq g_T(\lambda,\bm{u})\leq -C_1.\end{equation} Indeed, by the subsequent changes of variables $\widetilde{\bm{y}} = \bm{y}-\bm{u}$ and $\bm{y} = ({(A_T\lambda)^{\frac{1}{d+1}}}/{T})\widetilde{\bm{y}}$ in~\eqref{th7:gt}, one obtains
\[ g_T(\lambda,\bm{u}):=-\frac{1}{2}\sum\limits_{i=1}^{k-1}a_is_i  \int\limits_{\Delta_{\bm{u}}\big(ct_{i+1}^{\frac{1}{d}}\big) \setminus \Delta_{\bm{u}}\big(ct_{i}^{\frac{1}{d}}\big)} e^{-||\bm{y}||}d\bm{y}, \]
where $c = T/(A_T\lambda)^{\frac{1}{d+1}}.$ Therefore, since $\bm{u}\in\widetilde{\Delta},$ there exists $i\in \{1,2,...,k\},$ such that $\mathbb{0}\in \Delta_{\bm{u}}\big(ct_{i+1}^{\frac{1}{d}}\big) \setminus \Delta_{\bm{u}}\big(ct_{i}^{\frac{1}{d}}\big).$ The upper bound in \eqref{th7:gt_const} follows as  $c\in[1,\infty)$ for $\lambda\leq{T^{d+1}}/{A_T}.$  Replacing $\Delta_{\bm{u}}\big(ct_{i+1}^{\frac{1}{d}}\big) \setminus \Delta_{\bm{u}}\big(ct_{i}^{\frac{1}{d}}\big)$ by $\mathbb{R}^d$ gives the lower bound in \eqref{th7:gt_const}.

Let us define
\begin{equation}
\begin{split}\label{th4:kappa_over} \sup\limits_{a\in[-C_2, -C_1]} \mathcal{K}_1(a\lambda)  \leq  C\lambda^2\int\limits_{|y|\leq\frac{1}{C_1|\lambda|}}y^2W(dy) = C{\mathcal{K}}_1(C_1\lambda):=\overline{\mathcal{K}}_1(\lambda),  \\
\sup\limits_ {a\in[-C_2, -C_1]}\mathcal{K}_2(a\lambda)  \leq C|\lambda|\int\limits_{|y|>\frac{1}{C_2|\lambda|}}|y|W(dy) = C{\mathcal{K}}_2(C_2\lambda):=\overline{\mathcal{K}}_2(\lambda).
\end{split}
\end{equation}

By \eqref{th4:kappa_est}, \eqref{th7:gt_const} and \eqref{th4:kappa_over}, one gets
\[ |\mathcal{K}(-\lambda g_T(\lambda,\bm{u}))| \max\left( \lambda^{-\delta}, \lambda^\delta \right)\lambda^{-\frac{\alpha}{d+1}-1} \mathbb{1}\left(\lambda\leq\frac{T^{d+1}}{A_T}\right)\] \begin{equation}\label{th4:eq_integrability1}
\leq (\overline{\mathcal{K}}_1(\lambda)+\overline{\mathcal{K}}_2(\lambda))\max\left( \lambda^{-\delta},\lambda^{\delta} \right)\lambda^{-\frac{\alpha}{d+1}-1},\end{equation} where, for sufficiently small $\delta,$  the last function is integrable by \eqref{th4:eqkappa1}, \eqref{th4:eqkappa2} and \eqref{th4:kappa_over}.

Let us choose $\gamma\in(\max(1,\beta_{BG},(2\alpha-1)/(2(d+1))),\alpha/(d+1)).$ Then, one obtains
\[ |\mathcal{K}(\lambda)| \leq \int\limits_{|\lambda y|\leq1}|e^{i \lambda y} - 1 - i\lambda y|W(dy) + \int\limits_{|\lambda y|>1}|e^{i \lambda y} - 1 - i \lambda y|W(dy) \]
\begin{equation}\label{th4:eq_kappa_estimate}
\leq \int\limits_{|\lambda y|\leq1}|\lambda y|^2 W(dy) + \int\limits_{|\lambda y| > 1}(2 + |\lambda y|)W(dy) \leq \int\limits_{|\lambda y|\leq1} |\lambda y|^\gamma W(dy) + 3\int\limits_{|\lambda y|>1} |\lambda y|^\gamma W(dy) \leq C\lambda^\gamma.\end{equation} 

Note that for fixed $a_i,s_i, \in\mathbb{R},\ t_i\in[0,1], \ i=1,...,k-1,$ and all $\lambda>{T^{d+1}}/{A_T}$ and $\bm{u}\in B(2),$ for some $C>0$ it holds $g_T(\lambda,\bm{u})<C.$  Thus, by  \eqref{th4:eq_kappa_estimate}, one gets
\[ |\mathcal{K}\left(- \lambda g_T(\lambda,\bm{u}) \right)|  \max( \lambda^{-\delta}, \lambda^\delta) \lambda^{-\frac{\alpha}{d+1}-1}\mathbb{1}\left(\lambda>\frac{T^{d+1}}{A_T}\right) \]\[ \leq C\lambda^\gamma \max( \lambda^{-\delta}, \lambda^\delta) \lambda^{-\frac{\alpha}{d+1}-1}\mathbb{1}\left(\lambda>\frac{T^{d+1}}{A_T}\right) \leq C\lambda^{\gamma+\delta-\frac{\alpha}{d+1}-1}\mathbb{1}\left(\lambda> 1\right),\] where the last estimate follows from ${T^{d+1}}/{A_T}\to\infty,$ when $T\to\infty.$ As for sufficiently small~$\delta$ the upper bounds above and in \eqref{th4:eq_integrability1} are integrable, the integrand in $I_1(T)$ is bounded by an integrable function. Also, the integrand in $I_1(T)$ converges pointwise to that in \eqref{th4:lim}, when $T\to\infty.$ Therefore, the conditions of the dominated convergence theorem are satisfied.

Let us prove that the integral $I_2(T)$ converges to $0,$ when $T\to\infty.$ Due to \eqref{th4:eq_bruin_potter}, \eqref{th4:eq_kappa_estimate} and as $a_i,s_i, \ i=1,...,k,$ are finite and fixed, $|I_2(T)| $ is bounded by
\[ C \sum\limits_{i=1}^{k-1} \int\limits_0^{\infty} \int\limits_{B(2)\setminus \widetilde{\Delta}} \left( \frac{T^d}{2(A_T\widetilde{\lambda})^{\frac{d}{d+1}}}\ \int\limits_{\Delta(t_{i+1}^{\frac{1}{d}}) \setminus \Delta(t_{i}^{\frac{1}{d}})} e^{-\frac{T||\bm{y}-\bm{u}||}{(A_T\widetilde{\lambda})^{\frac{1}{d+1}}}}d\bm{y} \right)^\gamma d\bm{u} \max(\widetilde{\lambda}^{\delta},\widetilde{\lambda}^{-\delta})\, \widetilde{\lambda}^{\gamma-\frac{\alpha}{d+1}-1}d\widetilde{\lambda}.\] By the change of variables $\widetilde{\lambda} = T^{d+1}\lambda/A_T$ and as $\Delta(t_{i+1}^{\frac{1}{d}}) \setminus \Delta(t_{i}^{\frac{1}{d}})\in\widetilde{\Delta}, \ i=1,...,k-1,$ the left-hand side expression above is bounded by
\[C\left( \frac{A_T}{T^{d+1}} \right)^{\frac{\alpha}{d+1}-\gamma-\delta} \int\limits_0^\infty\int\limits_{B(2)\setminus \widetilde{\Delta}} \left({\lambda^{-\frac{d}{d+1}}} \int\limits_{\widetilde{\Delta}} e^{-\frac{||\bm{y}-\bm{u}||}{\lambda^{\frac{1}{d+1}}}}d\bm{y} \right)^\gamma d\bm{u} \max({\lambda}^{\delta},{\lambda}^{-\delta}) {\lambda}^{\gamma-\frac{\alpha}{d+1}-1}d{\lambda} \]\begin{equation}\label{eq:th4_i2new_sum} =: C\left( \frac{A_T}{T^{d+1}} \right)^{\frac{\alpha}{d+1}-\gamma-\delta} (I_{2,1} + I_{2,2}).\end{equation} where the first two integration sets in $I_{2,1}$ and $ I_{2,2}$ are $[0,1]\times (B(2)\setminus \widetilde{\Delta})$ and $[1,\infty)\times (B(2)\setminus \widetilde{\Delta})$ respectively.

Let us consider the integral $I_{2,1}.$ By the change of variables $\widetilde{\bm{y}} = \bm{y} - \bm{u},$ one gets
\[\int\limits_{0}^1 \int\limits_{B(2) \setminus {\widetilde{\Delta}}} \left(  \int\limits_{\widetilde{\Delta}_{\bm{u}}} e^{-\frac{||\widetilde{\bm{y}}||}{\widetilde{\lambda}^{\frac{1}{d+1}}}}d\widetilde{\bm{y}} \right)^\gamma d\bm{u} {\lambda}^{\gamma-\frac{\gamma d}{d+1}-\frac{\alpha}{d+1}-\delta-1}d{\lambda}\]\[\leq \int\limits_{0}^1 \int\limits_{B(2) \setminus {\widetilde{\Delta}}} \left(  \int\limits_{||\widetilde{\bm{y}}||\geq{\rm dist}({\bm{u}},\widetilde{\Delta})} e^{-\frac{||\widetilde{\bm{y}}||}{\widetilde{\lambda}^{\frac{1}{d+1}}}}d\widetilde{\bm{y}} \right)^\gamma d\bm{u} {\lambda}^{\gamma-\frac{\gamma d}{d+1}-\frac{\alpha}{d+1}-\delta-1}d{\lambda},\]where ${\rm dist}({\bm{u}},\widetilde{\Delta})$ denotes the distance between a point $\bm{u}$ and the set $\widetilde{\Delta}.$ By the change of variables $\bm{y} = \widetilde{\bm{y}}/\lambda^\frac{1}{d+1},$ one obtains
\[ \int\limits_{0}^1 \int\limits_{B(2) \setminus {\widetilde{\Delta}}} \left( \int\limits_{\lambda^{\frac{1}{d+1}}||\bm{y}||\geq{\rm dist}({\bm{u}},\widetilde{\Delta})} e^{-||\bm{y}||}d\bm{y} \right)^\gamma d\bm{u} {\lambda}^{\gamma-\frac{\alpha}{d+1}-\delta-1}d{\lambda}\]
\begin{equation}\label{eq:th4_int_2}=\int\limits_{0}^1 \int\limits_{B(2) \setminus {\widetilde{\Delta}}}\Gamma^\gamma\left(d, \frac{{\rm dist}(\bm{u}, \widetilde{\Delta})}{\lambda^{\frac{1}{d+1}}} \right)d\bm{u}{\lambda}^{\gamma-\frac{\alpha}{d+1}-\delta-1}d\lambda.\end{equation} Let us consider the inner integral above
\[ \int\limits_{B(2) \setminus {\widetilde{\Delta}}}\Gamma^\gamma\left(d, \frac{{\rm dist}(\bm{u}, \widetilde{\Delta})}{\lambda^{\frac{1}{d+1}}} \right)d\bm{u} = \int\limits_{B(2) \setminus \Delta_1}\Gamma^\gamma\left(d, \frac{{\rm dist}(\bm{u}, \partial\Delta_1)}{\lambda^{\frac{1}{d+1}}} \right)d\bm{u}\]\[+\int\limits_{ \Delta_2}\Gamma^\gamma\left(d, \frac{{\rm dist}(\bm{u}, \partial\Delta_2)}{\lambda^{\frac{1}{d+1}}} \right)d\bm{u}=:I_{2,1,1}+I_{2,1,2},\] where $\Delta_1=\Delta(t_k^\frac{1}{d})$ and $\Delta_2=\Delta(t_1^\frac{1}{d}).$ It holds \[I_{2,1,1}= \int\limits_{B(2) \setminus {\Delta_1(1+\lambda^{\frac{1}{2(d+1)}})}}\Gamma^\gamma\left(d, \frac{{\rm dist}(\bm{u}, \partial\Delta_1)}{\lambda^{\frac{1}{d+1}}} \right)d\bm{u}+\int\limits_{{\Delta_1(1+\lambda^{\frac{1}{2(d+1)}})} \setminus \Delta_1}\Gamma^\gamma\left(d, \frac{{\rm dist}(\bm{u}, \partial\Delta_1)}{\lambda^{\frac{1}{d+1}}} \right)d\bm{u}.\] Note, that because $\Delta_1$ is a convex set and the origin is in its interior, there is a constant $C>0$ such that ${\rm dist}(\bm{u}, \partial\Delta_1)\geq C \lambda^{\frac{1}{2(d+1)}},$ for  all $\bm{u}\in B(2) \setminus {\Delta_1(1+\lambda^{\frac{1}{2(d+1)}})}$ in the first integral. Also, $\Gamma(d,\cdot)<C$ and the volume of the integration set in the second integral is less than $C((\lambda^{\frac{1}{2(d+1)}}+1)^d-1)\le C\lambda^{\frac{1}{2(d+1)}}.$ Therefore,  
\[I_{2,1,1}\le C_1\Gamma\left(d,\frac{C}{\lambda^{\frac{1}{2(d+1)}}}\right) + C_2\lambda^{\frac{1}{2(d+1)}} \leq C \lambda^{\frac{1}{2(d+1)}},\ {\rm when} \ \lambda\to0,\] where the last estimate holds as the incomplete Gamma function satisfies $\Gamma(d, x) \sim x^{d-1}e^{-x},$ when $x\to\infty.$ 

If $t_1=0,$ then $I_{2,1,2}=0.$ Otherwise, analogously to $I_{2,1,1},$ one gets 
\[\int\limits_{ \Delta_2}\Gamma^\gamma\left(d, \frac{{\rm dist}(\bm{u}, \partial\Delta_2)}{\lambda^{\frac{1}{d+1}}} \right)d\bm{u} = \int\limits_{ \Delta_2(1-\lambda^{\frac{1}{2(d+1)}})}\Gamma^\gamma\left(d, \frac{{\rm dist}(\bm{u}, \partial\Delta_2)}{\lambda^{\frac{1}{d+1}}} \right)d\bm{u} \] \[\hspace*{1cm} + \int\limits_{\Delta_2 \setminus \Delta_2(1-\lambda^{\frac{1}{2(d+1)}})}\Gamma^\gamma\left(d, \frac{{\rm dist}(\bm{u}, \partial\Delta_2)}{\lambda^{\frac{1}{d+1}}} \right)d\bm{u}\leq C\lambda^{\frac{1}{2(d+1)}}, \ {\rm when} \ \lambda\to0. \] 
Thus, the integral \eqref{eq:th4_int_2} is finite as $\gamma\in(\max(1,\beta_{BG},\frac{2\alpha-1}{2(d+1)}),\frac{\alpha}{d+1}),$ from which follows the finiteness of $I_{2,1}.$

Let us show that the integral $I_{2,2}$ is finite. As the inner integral is bounded, it holds 
\[ \int\limits_1^{\infty} \int\limits_{B(2)\setminus \widetilde{\Delta}} \left( \frac{1}{\lambda^{\frac{d}{d+1}}}\ \int\limits_{\widetilde{\Delta}} e^{-\frac{||\bm{y}-\bm{u}||}{\lambda^{\frac{1}{d+1}}}}d\bm{y} \right)^\gamma d\bm{u} {\lambda}^{\gamma-\frac{\alpha}{d+1}+\delta-1}d{\lambda}\leq C\int\limits_{1}^\infty \lambda^{\gamma+\delta-\frac{\alpha}{d+1}-1}<\infty,\]  as $\gamma<\alpha/(d+1)$ and $\delta$ can be chosen arbitrarily small.

Thus, the integrals $I_{2,1}$ and $I_{2,2}$ are finite and as $A_T/T^{d+1}\to0,$ when $T\to\infty,$ it follows from \eqref{eq:th4_i2new_sum} that $I_2(T)\to 0,$ when $T\to \infty.$

Let us prove that the integral $I_3(T)$ converges to $0,$ when $T\to\infty.$  It holds
\[|I_3(T)| \leq \int\limits_0^{\infty} \int\limits_{\mathbb{R}^d\setminus B(2)} |\mathcal{K}\left(- \lambda g_T(\lambda,\bm{u}) \right)|d\bm{u} \max( \lambda^{-\delta}, \lambda^\delta ) \lambda^{-\frac{\alpha}{d+1}-1}d\lambda \]
\[ = \int\limits_0^{\infty} \int\limits_{\mathbb{R}^d\setminus B(2)} |\mathcal{K}\left(- \lambda g_T(\lambda,\bm{u}) \right)|d\bm{u} \max( \lambda^{-\delta}, \lambda^\delta ) \lambda^{-\frac{\alpha}{d+1}-1}\mathbb{1}\left(\lambda>\frac{T^{d+1}}{A_T}\right)d\lambda \]
\[ + \int\limits_0^{\infty} \int\limits_{\mathbb{R}^d\setminus B(2)} |\mathcal{K}\left(- \lambda g_T(\lambda,\bm{u}) \right)|d\bm{u} \max( \lambda^{-\delta}, \lambda^\delta ) \lambda^{-\frac{\alpha}{d+1}-1}\mathbb{1}\left(\lambda\leq\frac{T^{d+1}}{A_T}\right)d\lambda \] \[: = I_3^{(1)}(T)+I_3^{(2)}(T).\] Let us consider the inner integral in $I_3^{(1)}(T).$ As $a_i,s_i\in\mathbb{R}, \ i=1,...,k-1,$ are fixed and $\Delta(t_i)\in B(1), \ i=1,...,k-1,$  by applying \eqref{th4:eq_kappa_estimate}, one gets
\begin{eqnarray} & &\hspace*{-1cm} \int\limits_{\mathbb{R}^d\setminus B(2)} |\mathcal{K}(-\lambda g_T(\lambda,\bm{u}))|d\bm{u} \leq C \lambda^\gamma \int\limits_{\mathbb{R}^d \setminus B(2)} |g_T(\lambda,\bm{u})|^\gamma d\bm{u} \nonumber\\
&\leq& C\lambda^\gamma\int\limits_{\mathbb{R}^d\setminus B(2)}\left( \frac{T^d}{(A_T\lambda)^{\frac{d}{d+1}}} \right)^\gamma \left( \int\limits_{B(1)} e^{-\frac{T||\bm{y}-\bm{u}||}{(A_T\lambda)^{\frac{1}{d+1}}}} d\bm{y} \right)^\gamma d\bm{u}\nonumber\\
&\leq& C \lambda^\gamma \left(  \frac{T^d}{(A_T\lambda)^{\frac{d}{d+1}}} \right)^\gamma e^{\frac{\gamma T}{(A_T\lambda)^{\frac{1}{d+1}}}}\int\limits_{\mathbb{R}^d\setminus B(2)} e^{-\frac{\gamma T||\bm{u}||}{(A_T\lambda)^{\frac{1}{d+1}}}}d\bm{u}\nonumber\\
& = & C\lambda^\gamma \left(  \frac{T}{(A_T\lambda)^{\frac{1}{d+1}}} \right)^{d\gamma-1} e^{\frac{\gamma T}{(A_T\lambda)^{\frac{1}{d+1}}}} \int\limits_{\frac{2\gamma T}{(A_T\lambda)^{\frac{1}{d+1}}}}^\infty u^{d-1}e^{-u}du \leq C\lambda^\gamma.\end{eqnarray} 
The last estimate holds as the integral represents the incomplete Gamma function $\Gamma(d, \cdot),$ which is bounded, and ${T}/(A_T\lambda)^{\frac{1}{d+1}}\in(0,1)$ in $I_3^{(1)}(T).$ 
Therefore,
\[I_3^{(1)}(T) \leq C \int\limits_\frac{T^{d+1}}{A_T}^\infty\lambda^{\gamma+\delta-\frac{\alpha}{d+1}-1}d\lambda\to0, \ {\rm when} \ T\to\infty, \] as $\gamma<{\alpha}/(d+1)$ and $\delta>0$ can be chosen arbitrarily small.

Let us show that the integral $I_3^{(2)}(T)$ converges to $0,$ when $T\to\infty.$ As $a_i,s_i\in\mathbb{R},$ $i=\overline{1,k-1},$ are fixed and $\Delta(t_i)\in B(1), \ i=\overline{1,k},$ by applying \eqref{th4:eq_kappa_estimate}, one gets
\[|I_3^{(2)}(T)|\leq\int\limits_0^{\frac{T^{d+1}}{A_T}} \int\limits_{\mathbb{R}^d\setminus B(2)} \left|\mathcal{K}\left(- \lambda g_T(\lambda,\bm{u}) \right)\right| d\bm{u} \max( \lambda^{-\delta}, \lambda^\delta ) \lambda^{-\frac{\alpha}{d+1}-1}d\lambda  \]
\[ \leq C  \int\limits_0^{\frac{T^{d+1}}{A_T}} \int\limits_{\mathbb{R}^d\setminus B(2)} \left(  \frac{T^d}{(A_T\lambda)^{\frac{d}{d+1}}} \int\limits_{B(1)}e^{-\frac{T||\bm{y}-\bm{u}||}{(A_T\lambda)^{\frac{1}{d+1}}}}d\bm{y} \right)^\gamma d\bm{u} \max(\lambda^{-\delta}, \lambda^{\delta}) \lambda^{\gamma-\frac{\alpha}{d+1}-1}d\lambda.\] By the change of variables $\lambda = \frac{T^{d+1}}{A_T}\widetilde{\lambda},$ the above double integral is bounded by
\begin{equation}\label{th4:eqj2}  \left( \frac{A_T}{T^{d+1}} \right)^{\frac{\alpha}{d+1}-\gamma-\delta} \int\limits_0^{1}\int\limits_{\mathbb{R}^d\setminus B(2)} \left( \frac{1}{\widetilde{\lambda}^\frac{d}{d+1}}\int\limits_{B(1)} e^{-\frac{||\bm{y}-\bm{u}||}{\widetilde{\lambda}^{\frac{1}{d+1}}}} d\bm{y} \right)^\gamma d\bm{u}\, \widetilde{\lambda}^{\gamma-\frac{\alpha}{d+1}-\delta-1}d\widetilde{\lambda}.\end{equation} The inner integral in \eqref{th4:eqj2} can be estimated from above as
\[ \int\limits_{\mathbb{R}^d\setminus B(2)} \left( \frac{1}{\widetilde{\lambda}^\frac{d}{d+1}} \int\limits_{B(1)} e^{-\frac{||\bm{y}-\bm{u}||}{\widetilde{\lambda}^{\frac{1}{d+1}}}} d\bm{y}  \right)^\gamma d\bm{u}  \leq \frac{C}{\widetilde{\lambda}^{\frac{d\gamma}{d+1}}}e^{\frac{\gamma}{\widetilde{\lambda}^{\frac{1}{d+1}}}} \int\limits_2^\infty u^{d-1} e^{-\frac{\gamma u}{\widetilde{\lambda}^{\frac{1}{d+1}}}}du.\] By the change of variables ${\gamma}{\widetilde{\lambda}^{-\frac{1}{d+1}}}u=
\widetilde{u},$ the expression on the right-hand side above is equal to
\[  \frac{C}{\widetilde{\lambda}^{\frac{d(\gamma-1)}{d+1}}\gamma^d} e^{\frac{\gamma}{\widetilde{\lambda}^{\frac{1}{d+1}}}} \int\limits_{\frac{2\gamma}{\widetilde{\lambda}^{\frac{1}{d+1}}}}^\infty \widetilde{u}^{d-1} e^{-\widetilde{u}}d\widetilde{u}\, \leq\, \frac{C}{\widetilde{\lambda}^{\frac{d\gamma-1}{d+1}}}e^{-\frac{\gamma}{\widetilde{\lambda}^{\frac{1}{d+1}}}}, \ \mbox{when} \ \widetilde{\lambda}\to0, \] where the upper bound follows from the asymptotic behaviour of the incomplete Gamma function $\Gamma(d,x)\sim x^{d-1}e^{-x},\ x~\to~\infty.$ 

Thus, the integral in \eqref{th4:eqj2} is finite, and $I_3^{(2)}(T)\to 0,$  as ${A_T}/{T^{d+1}}\to 0,$ when $T\to\infty,$ which completes the proof. \end{proof}

\section{Conclusions and future work} 
The paper introduced a new class of supCAR random fields constructed as superpositions of continuous autoregressive random fields. These fields possess infinitely divisible marginal distributions and new types of covariance functions. By specifying the superpositions, one can get different local behaviours of supCAR fields, making their realisations smoother or rougher, and, at the same time, exhibiting short- or long-range large-scale dependence. The asymptotics of integrated supCAR fields give four different limiting processes, depending on the decay of the correlation function and their marginal distributions: 
\begin{itemize}
    \item[-] For a short-range dependent supCAR field $X(\cdot)$ ($\pi(\cdot)$ has a density $p(\lambda)=l(1/\lambda)\lambda^{\alpha-1},\\ \alpha>2d+2$),  the limit process is the standard Brownian motion (Theorem \ref{th:brow}). 
    \item[-] For a long-range dependent supCAR field $X(\cdot)$ ($\alpha\leq 2d+2$) containing the Gaussian component in the underlying CAR(1) fields ($b>0$ in \eqref{prop:cum_func_mom}), the limit is the generalised Brownian motion (Theorem~\ref{th:herm}). 
    \end{itemize}
For a long-range dependent supCAR field $X(\cdot)$ that does not contain the Gaussian component in the underlying CAR(1) fields ($b=0$ in \eqref{prop:cum_func_mom}), there are two possible limit scenarios depending on the asymptotics of the L\'evy measure $W(\cdot)$ at the origin. 
\begin{itemize}
    \item[-] If the Blumental-Getour index of $W(\cdot)$ satisfies $\alpha/(d+1)<\beta_{BG}<\min(2, \alpha-d),$ then the limit is the $\gamma$-process (Theorem \ref{th6}).
    \item[-] For $0<\beta_{BG}<\alpha/(d+1)$ the limit is the $\alpha$-stable L\'evy process (Theorem \ref{th7}).
\end{itemize}
Figure~\ref{fig_sum} visualises the obtained results. 
\begin{figure}[htb!]
  \centering
  \includegraphics[width=0.6\linewidth, height=0.4\linewidth]{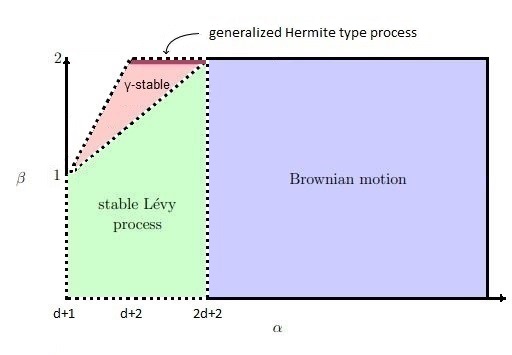}\vspace{-3mm}
  \caption{Parameter space diagram of possible limit scenarios}
  \label{fig_sum}
\end{figure}

In future research, it would be interesting to investigate asymptotic behaviours of the integrated supCAR fields when $\beta_{BG} = \alpha/(d+1)$ and obtain weak convergence results in the
space~$C[0,1].$ It is expected that the weak convergence in the space $C[0,1]$ holds is Theorems~\ref{th:brow} and~\ref{th:herm} under the same conditions, and in Theorem \ref{th6} under additional conditions on moments of the supCAR field $X(\cdot),$ while Theorem \ref{th7}  may not have extensions for the weak convergence in the space of c\`adl\`ag functions $D[0,1]$ equipped with Skorokhod's $J_1$ topology. It would also be interesting to investigate the behaviour of moments and other functionals of the integrated supCAR fields.

Another natural future direction is to investigate whether supCAR fields exhibit intermittency, particularly within the long-range dependence regime. It is reasonable to expect such behaviour, as intermittency effects frequently emerge in spatial settings. To establish the presence of intermittency, one may examine the growth rates of the moments in the considered limiting regimes. It is anticipated that, in analogy with the one-dimensional models \cite{grahovac2019limit}, intermittency will manifest through changes in these rates.

\section*{Acknowledgements} This research was supported by the Australian Research Council's Discovery Projects funding scheme (project DP220101680). I.~Donhauzer and A.~Olenko were partially supported by La Trobe University's SCEMS CaRE and Beyond grant. N. Leonenko was partially supported under the Croatian Scientific Foundation (HRZZ) grant “Scaling in Stochastic Models” (IP-2022-10-8081), grant FAPESP 22/09201-8 (Brazil) and the Taith Research Mobility grant (Wales, Cardiff University, 2025).

%%%%%%%%%%%%%%%%%%%%%%%%%%%%%%%%%%%%%%%%%%%%%%%%%%%%%%%%%%%%%%%%%%%
%%                                                               %%
%% Use the two commands below for producing your bibliography    %%
%% with bibtex, then comment again the commands and include the  %%
%% content of the .bbl file in this file below the commands.     %%
%%                                                               %%
%%%%%%%%%%%%%%%%%%%%%%%%%%%%%%%%%%%%%%%%%%%%%%%%%%%%%%%%%%%%%%%%%%%

%\bibliographystyle{amsplain}
%\bibliography{yourbibfilename}

% add below the content of your .bbl file produced by bibtex.

\end{document}